\documentclass[11pt]{article}
\usepackage[activeacute,english]{babel} %Languages
\usepackage[utf8]{inputenc}
\usepackage{verbatim} %allows one to comment large chunks
\usepackage{graphicx} %For inserting images
\usepackage[usenames,dvipsnames]{color} %For colors
\usepackage{enumerate,mdwlist}  %For list and personalization of lists
\usepackage{multirow} %For the tables
\usepackage{amsmath,amsfonts,amssymb,amsthm} %AMSpackages
\usepackage{bm} %Bold math symbols
\usepackage{empheq} %For braces in systems of equations
\usepackage{bbm,dsfont} %Blackboard symbols
\usepackage{mathrsfs} % script letters
\numberwithin{equation}{section} %Numbering of equations
\usepackage{float} % for the position of the figures
\usepackage[labelfont=bf,labelsep=space]{caption} % for captions
\usepackage{empheq} %for systems with braces

\usepackage{xspace}
\usepackage{todonotes}

\usepackage[OT2,OT1]{fontenc}
\DeclareRobustCommand\cyr{%
  \renewcommand\rmdefault{wncyr}%
  \renewcommand\sfdefault{wncyss}%
  \renewcommand\encodingdefault{OT2}%
  \normalfont
  \selectfont}
\DeclareTextFontCommand{\textcyr}{\cyr}

\usepackage{subcaption}

\usepackage{geometry}
\usepackage{color}
\definecolor{red}{rgb}{.7,0,0}
\definecolor{blue}{rgb}{0,0,1}

\def\mcP{\mathcal{P}}

\def\mcS{\mathcal{S}}
\def\mcT{\mathcal{T}}

\def\mcZ{\mathcal{Z}}

\def\bbD{\mathbb{D}}
\def\bbF{\mathbb{F}}

\def\bbR{\mathbb{R}}
\def\bbZ{\mathbb{Z}}
\def\bbN{\mathbb{N}}
\def\bbI{\mathbb{I}}
\def\bbQ{\mathbb{Q}}
\def\bbC{\mathbb{C}}
\def\bbP{\mathbb{P}}
\def\bbE{\mathbb{E}}

\def\fkc{\mathfrak{c}}

\def\fkc{\mathfrak{c}}
\def\fkf{\mathfrak{f}}

\def\fkx{\mathfrak{x}}
\def\fky{\mathfrak{y}}
\def\fkz{\mathfrak{z}}

\newcommand{\uObD}{{\overline{\mathcal{D}}}\xspace}
\newcommand{\dObD}{{\underline{\mathcal{D}}}\xspace}
\newcommand{\ObL}{\mathcal{L}\xspace}
\newcommand{\ObA}{\mathcal{A}\xspace}

\newcommand\xqed[1]{%
  \leavevmode\unskip\penalty9999 \hbox{}\nobreak\hfill
  \quad\hbox{#1}}
\newcommand{\eproof}{\xqed{\qed}}
\usepackage{ifpdf}
\ifpdf
\usepackage[pdfencoding=auto,unicode,psdextra]{hyperref} 
\hypersetup{
 pdftoolbar=true,
 pdfmenubar=true,
 pdfstartview={FitV},
 pdftitle={A p-adic Descartes solver: the Strassman solver},
 pdfauthor={Josu\'{e} Tonelli-Cueto},%, Elias Tsigaridas},
 pdfsubject={probability, condition number, polynomial systems,p-adics},
 pdfkeywords={algebraic geometry,polynomial suystems,p-adics},
 pdfcreator={overleaf}, 
 pdfproducer={overleaf},
 linkcolor=red     %Color de los links internos del PDF.
 citecolor=green   %Color de las citas.
 urlcolor=cyan     %Color de los links externos.
 filecolor=magenta %Color de los links externos a archivos.
}
\fi
\usepackage{bookmark}

\title{A $p$-adic Descartes solver: the Strassman solver}
\author{Josu\'{e} Tonelli-Cueto\thanks{Supported by a postdoctoral fellowship of the 2020 ``Interaction'' program of the Fondation Sciences Mathématiques de Paris. Partially supported by ANR JCJC
GALOP (ANR-17-CE40-0009), the PGMO grant ALMA, and the PHC GRAPE.}\\
Inria Paris \& IMJ-PRG\\ 
Sorbonne Université\\ 
Paris, FRANCE\\
{\tt  josue.tonelli.cueto@bizkaia.eu}
}
\makeatletter
\def\th@plain{%
  \thm@notefont{}% same as heading font
  \slshape % body font
}
\def\th@definition{%
  \thm@notefont{}% same as heading font
  \normalfont % body font
}
\makeatother
\theoremstyle{plain}
\newtheorem{lem}{Lemma}[section]
\newtheorem{prop}[lem]{Proposition}
\newtheorem{theo}[lem]{Theorem}
\newtheorem{theoM}{Theorem}
\newtheorem{cor}[lem]{Corollary}

\theoremstyle{definition}
\newtheorem{defi}[lem]{Definition}
\theoremstyle{remark}

\newtheorem{problem}{Problem}
\newtheorem{exam}[lem]{Example}

\newtheorem{remark}[lem]{Remark}
\def\bfd{\boldsymbol{d}}

\def\Pd{\mcP_{n,\bfd}}

\def\Oh{\mathcal{O}}

\DeclareMathOperator{\dist}{dist}

\def\Oh{\mathcal{O}}
\def\diff{\mathrm{D}}
\def\newton{\mathrm{N}}

\DeclareMathOperator{\St}{\mathrm{St}}
\DeclareMathOperator{\V}{\mathrm{V}}
\def\enumber{\mathrm{e}}

\def\Qp{\bbQ_p}
\def\Zp{\bbZ_p}
\def\Cp{\bbC_p}

\usepackage{multicol}
\usepackage[linesnumbered,ruled,algosection,vlined]{algorithm2e}
\SetKwInOut{postcondition}{Postcondition}
\SetKwInOut{precondition}{Precondition}
\SetKwInOut{correctness}{Successful if}

\SetCommentSty{mycommfont}
\let\oldnl\nl% Store \nl in \oldnl
\newcommand{\nonl}{\renewcommand{\nl}{\let\nl\oldnl}}% Remove line number for one line

\makeatletter
\let\original@algocf@latexcaption\algocf@latexcaption
\long\def\algocf@latexcaption#1[#2]{%
  \@ifundefined{NR@gettitle}{%
    \def\@currentlabelname{#2}%
  }{%
    \NR@gettitle{#2}%
  }%
  \original@algocf@latexcaption{#1}[{#2}]%
}
\makeatother

\usepackage{rotating,lipsum}

\begin{document}
\date{}
\maketitle

\begin{abstract}
Solving polynomials is a fundamental computational problem in mathematics. In the real setting, we can use Descartes' rule of signs to efficiently isolate the real roots of a square-free real polynomial. In this paper, we translate this method into the $p$-adic worlds. We show how the $p$-adic analog of Descartes' rule of signs, Strassman's theorem, leads to an algorithm to isolate the roots of a square-free $p$-adic polynomial. Moreover, we show that this algorithm runs in $\mathcal{O}(d^2\log^3d)$-time for a random $p$-adic polynomial of degree $d$. To perform this analysis, we introduce the condition-based complexity framework from real/complex numerical algebraic geometry into $p$-adic numerical algebraic geometry.
\end{abstract}
\section{Introduction}

Analogies and comparison between the real and $p$-adic worlds are both a bless and a curse. On the one hand, it inspires us to translate results from on world to the other; on the other hand, this translation is not always obvious and we might loss the intuition in the translation process. Yet, this allows us to see better how different and how similar these worlds are.

An example of such a translation is fewnomial theory. In the real world, Khovanskii~\cite{fewnomialbook} showed that the number of isolated real zeros of a real polynomial system can be bounded solely in terms of the number of variables and the number of monomial terms of the system. In the $p$-adic worlds, Rojas~\cite{rojas2001} produced an analogous bound for the number of isolated $p$-adic zeros. In both cases, the big open problem is to obtain bounds that are polynomial in the number of monomials~\cite{phillipsonrojas2013}.

Recently, there has been a surge of interest in translating the results from random real algebraic geometry, e.g.~\cite{edelmankostlan1995}, to the $p$-adic world to create a random $p$-adic algebraic geometry~\cite{caruso2021,kulkarnilerario2021} that goes beyond the seminal result of Evans~\cite{evans2006}.

In this paper, we aim to contribute to the translation of random real algebraic geometry into $p$-adic algebraic geometry by translating the condition-based complexity framework from real/complex numerical algebraic geometry to $p$-adic numerical algebraic geometry. We will illustrate this framework with an algorithm based in a $p$-adic analogue of Descartes' rules of signs: Strassman's theorem.

\subsection{Numerical algorithms and condition-based complexity}

In the real/complex world, the complexity of a numerical algorithm is not uniform, the algorithm might need more computational resources---run-time or precision---for some input than for others. The condition-based framework~\cite{conditionbook} (initiated by Turing~\cite{turing1948} and von Neumann and Goldstine~\cite{vonneumanngoldstine1947}) is based on the fact that the computational cost of processing an input depends on the \emph{condition number} of this input, which is a measure of the numerical sensitivity of the input for the problem we are trying to solve. The bigger the condition number of an input is, the bigger the effect of small perturbations of the input in the solution is, and so the bigger the computational resources needed to handle this input are.

In this way, the condition-based framework of complexity aims to understand the complexity of numerical algorithms in terms of the size and condition number of the input. However, although effective for understanding how a numerical algorithm behaves at a particular input, condition-based estimates don't necessarily give an idea of how a numerical algorithm behaves in general. 

To go beyond input-dependent complexity estimate, we randomize the input to study how the algorithm behaves statistically for a random input. This idea, which goes back to Goldstine and von Neumann~\cite{goldstinevonneumann1951}, Demmel~\cite{demmel1987,demmel1988} and Smale~\cite{smale1997}, is the key to transform input-dependent condition-based complexity estimates into input-independent probabilistic ones. Moreover, one can consider the smoothed framework~\cite{spielmanteng2002}, in which we consider an arbitrary input perturbed by random noise, to get a more realistic estimate of the behaviour of an algorithm in practice. 

In the $p$-adic worlds, we can tell the same story as above. Unfortunately, up to the knowledge of the author, there is not an analog condition-based framework. However, there are probabilistic complexity analyses~\cite{caruso2015,caruso2017} for the precision of numerical algorithm in $p$-adic linear algebra, where it is common to consider experiments based on random inputs. This paper fills the gap by illustrating the condition-based framework for a novel algorithm for solving univariate $p$-adic polynomials. We note that the advantage of this complexity framework relies on the fact that for many problems it explains the behaviour of the algorithms better than the worst-case bit-complexity framework, where we bound the worst possible complexity in terms of the bit-size of the input.

In section~\ref{sec:condition}, we introduce condition numbers for solving univariate $p$-adic polynomial in $\Zp$, by adapting the techniques in~\cite{TCTcubeI-journal} (cf.~\cite{TCTcubeI}), which are based in the condition number for solving real polynomial systems~\cite{CKMW1}. In section~\ref{sec:random}, we translate the probabilistic techniques of the real setting to the $p$-adic one. For this, we adapth the techniques of Erg\"ur, Paouris and Rojas~\cite{EPR2018,EPR2019} which are based on~\cite{CKMW3} and geometric functional analysis~\cite{vershyninbook}. In the end, in Section~\ref{sec:complexity}, we show how all these results are applied to analyze the complexity of an algorithm: \nameref{alg:strassman}.

\begin{remark}
In this work, we only give the first steps towards a condition-based complexity framework in the $p$-adic worlds, so we focus on the average complexity analysis. We leave for future work to develop the smoothed probabilistic model, where we consider an arbitrary $p$-adic polynomial perturbed by random noise.
\end{remark}

\subsection{The \nameref{alg:strassman} solver}

The Descartes' rule of signs (see Theorem~\ref{theo:descartesruleofsigns}) allows us to bound the number of real roots of a univariate only in terms of the sign variations of its coefficients. A famous corollary of this is that the number of isolated real roots of a real univariate polynomial is linear in the number of monomials. The latter was generalized to the $p$-adic setting by Lenstra~\cite{lenstra1999}.

Now, the generalization of Lenstra~\cite{lenstra1999} is not a direct generalization of the Descartes' rule of signs, but of its famous corollary. As the sign is the discrete valuation of $\bbR$, we can ask the following: is there a bound on the number of $p$-adic roots of a $p$-adic polynomial that only depends on the $p$-adic valuation of the coefficients? The answer is yes: Strassman's theorem (Theorem~\ref{theo:strassmanbound}).

The analogy between Descartes' rule of signs and Strassman's theorem does not end here. We can establish many parallelisms as we will show in Section~\ref{sec:descartesvsstrassman}. Among the most important one, we have that in the same way that Descartes' rule of signs leads to a univariate solver for the reals, \nameref{alg:descartes}, Strassman's theorem lead to a univariate solver for the $p$-adics: \nameref{alg:strassman}.

Imitating the condition-based analyses for \nameref{alg:descartes} in \cite{TCTcubeI-journal} and \cite{ETCT-descartes}, we provide such an analysis for \nameref{alg:strassman}. We show the following:

\begin{theoM}\label{theo:maintheorem}
Let $\fkf=\sum_{k=0}^d\fkf_kT^k\in\Zp[T]$ be a random $p$-adic polynomial of degree $d$, i.e., the $\fkf$ are independent random $p$-adic variables uniformly distributed in $\Zp$. Then the algorithm \nameref{alg:strassman} finds an approximations of all roots of $\fkf$ in $\Zp$ using $\Oh(d^2\log^3d\log p)$ arithmetic operations on the average. Furthermore, the precision needed by \nameref{alg:strassman} in the average to guarantee correctness is $d+\Oh(1)$.

Moreover, if $p\leq \Oh(d)$, then the average number of arithmetic operations can be reduced to $\Oh(dp)\leq \Oh(d^2)$.
\end{theoM}
\begin{remark}
By ``approximations of all roots of $\fkf$ in $\Zp$'', we mean that the Newton method---Hensel's lifting---starting at these approximations converge quadratically. A precise convergence statement is given in Proposition~\ref{prop:strassmantoalpha}.
\end{remark}
\begin{remark}
Even though the average precision is $d+\Oh(1)$, one can see from the proofs that this is only needed at the beginning of the algorithm. Afterwards the average precision goes down to $\Oh(1)$. 
\end{remark}

In the precision analysis of \nameref{alg:strassman}, we use a flat precision model~\cite{carusoroevaccon2016} where all numbers involved are written with the same precision. In the future, it would be interesting to see how \nameref{alg:strassman} behave under more sophisticated precision analyses such as those in~\cite{carusoroevaccon2014,carusoroevaccon2016}.

We note that \nameref{alg:strassman} is what we can call a subdivision algorithm. In the real world, these algorithms are quite extensive (see~\cite{yap2019towards}); but they are underexplored compared to the so-called homotopy continuation---used in the solution Smale's 17th problem~\cite{lairez2017}. Breiding~\cite{breiding2013} made an attempt to generalize homotopy continuation methods, but the metric/topological properties of the $p$-adics made such an attempt fail. In contrast, subdivision methods are commonplace in the $p$-adic world~\cite{Dubhashi-phd-92,Loos-rootsQ-83,MalWhi-cell-decomp-99} and also in the related world of prime power rings~\cite{chengaorojas2019,kopprandallrojaszhu2020}. Nevertheless, none of these algorithms seems to use Strassman's theorem as the guiding rule of the subdivision, as \nameref{alg:strassman} does. A notable exception to subdivision-based method in $p$-adic polynomial system solving is \cite{kulkarni2020}, which uses $p$-adic linear algebra but no complexity analysis is provided.

We describe \nameref{alg:strassman} in Section~\ref{sec:descartesvsstrassman}. Then we provide a complexity analysis in Section~\ref{sec:complexity}, using the results in Section~\ref{sec:condition}; which we turn into a probabilistic analysis in Section~\ref{sec:random}.

\subsection[A p-adic Smale's 17th problem]{A $p$-adic Smale's 17th problem}

At the core of the classical Smale's 17th problem~\cite{smale2000}, we have the question of whether a random complex polynomial system can be solved fast? Over non-algebraically closed fields, we don't ask whether we can solve fast, but whether we determine feasibility fast. Given how fruitful Smale's 17th problem was for complex numerical algebraic geometry, we do the same in the $p$-adic setting with the objective of developing the condition-based framework in $p$-adic numerical algebraic geometry.

We state two versions. One for the random model that takes coefficients with respect to the monomial basis and one that it takes coefficients with respect to the binomial basis---considered already by Evans~\cite{evans2006}.

\begin{problem}[$p$-adic Smale's 17th Problem M]
Let $\fkf\in\Zp[X_1,\ldots,X_n]^n$ be a random $p$-adic polynomial system such that
\[\fkf_i=\sum_{|\alpha|\leq d_i}\fkf_{i,\alpha}X^\alpha\]
with the $\fkf_{i,\alpha}$ independent random $p$-adic variable uniformly distributed in $\Zp$. Is there a deterministic algorithm that decides whether or not $\fkf$ has a zero in $\Zp^n$ (resp. $\Qp^n$) in average polynomial-time with respect the number of coefficients?
\end{problem}
\begin{problem}[$p$-adic Smale's 17th Problem B]
Let $\fkf\in\Zp[X_1,\ldots,X_n]^n$ be a random $p$-adic polynomial system such that
\[\fkf_i=\sum_{|\alpha|\leq d_i}\fkc_{i,\alpha}\prod_{j=1}^n\binom{X_j}{\alpha_j}\]
with the $\fkc_{i,\alpha}$ independent random $p$-adic variable uniformly distributed in $\Zp$. Is there a deterministic algorithm that decides whether or not $\fkf$ has a zero in $\Zp^n$ (resp. $\Qp^n$) in average polynomial-time with respect the number of coefficients?
\end{problem}

We note that we can be more ambitious and consider also the sparse version as Rojas and Ye~\cite{rojasye2005} in the real world. We note that the results of \cite{avendanoibrahimrojasrusek2012,bichengrojas2016,rojaszhu2020} impose restrictions for an input that is not random, so the above problems might have a positive solutions. 

\begin{remark}
We note that \nameref{alg:strassman} does not solve Problem 1 for $n=1$, since \nameref{alg:strassman} relies on the Cantor-Zassenhaus factorization algorithm~\cite{cantorzassenhaus1981} which is not deterministic. 
\end{remark}

\subsection[Beyond Qp]{Beyond $\Qp$}

We note that the results in this paper can be generalized to the finite extensions of $\Qp$ in a reasonable way. However, for the sake of avoiding getting unnecessarily technical, we restrict to computations over $\Qp$.

\paragraph{Notation} $\Qp$ will denote the fiel of $p$-adic numbers and $\Zp$ the ring of $p$-adic integers. To denote the norm in them, we will simply use $|~|$. Similarly, we will denote by $\Cp$ the analytic closure of the algebrac closure of $\Qp$, denoting its absolute value also by $|~|$. We will als use $\|~\|$ for the corresponding norm of $p$-adic vectors and polynomials. To denote random variables we will use fraktur letters.

\paragraph{Acknowledgements} The author is grateful to Elias Tsigaridas for various discussions, suggestions and support; to Matías Bender for suggestions; and to Evgenia Lagoda for her constant moral support and Gato Suchen for a critical suggestion regarding the proof of Theorem~\ref{theo:conditionproperties}.

\section{Descartes vs. Strassman}\label{sec:descartesvsstrassman}

Given a real polynomial $f=\sum_{k=0}^df_kT^k\in\bbR[T]$, we can consider the number of signs variations of its list of coefficients:
\begin{equation}
    \V(f):=\{k\in\bbN\mid (f_k\geq 0\text{ and }f_{k+1}<0)\text{ or }(f_k\leq 0\text{ and }f_{k+1}>0)\}.
\end{equation}
Note that a sign change means that the
coefficient changes from 
positive to negative or negative to positive, i.e., we are counting sign changes in sequences where we omit the zeros. The so-called Descartes' rule of signs relates the number of positive roots of $f$ to the number of sign changes in the coefficient list.

\begin{theo}[Descartes' rule of signs]\label{theo:descartesruleofsigns}
Let $f=\sum_{k=0}^df_kT^k\in\bbR[T]$ be a real polynomial. Then
\[
\mcZ(f,\bbR_+)\leq \V(f).
\]
Moreover, we have equality if $\V(f)$ is zero or one.\eproof
\end{theo}

In particular, the difference between the actual number of positive roots and the
number of sign variations is always an even number.
Moreover, to count the real roots of $f$ in an interval
$I = (a, b)$
we use the transformation $x \mapsto \frac{a T + b}{T + 1}$
that maps $I$ to $(0, \infty)$.
Then
\begin{equation}
\label{eq:I-to-0inf}
	V(f, I) := V( (T + 1)^d f(\tfrac{a T + b}{T + 1}) )
\end{equation}
bounds the number of real roots of $f$ in $I$.

In $p$-adic analysis, there is a theorem with a similar flavour due to Strassman. In this case, the $\infty$-adic valuations play the role of signs. So Strassman's theorem is a $p$-adic analogue of Descartes' rule of signs in the sense that it gives a bound on the number of $p$-adic roots (in $\Zp$) in terms of the $p$-adic valuation of the coefficients.

\begin{theo}[Strassman's theorem]\label{theo:strassmanbound}
Let $f=\sum_{k=0}^df_kT^k\in\Qp[T]$ be a $p$-adic polynomial. Then
\[
\mcZ(f,\Zp)\leq \St(f):=\max\{k\in\bbN\mid\text{for all }l<k,\,|f_l|\leq |f_k|\}.
\]
Moreover, we have equality if $\St(f)$ is zero or one.
\end{theo}
\begin{proof}
The inequality is well-known; see for example~\cite[Theorem~4.4.6]{gouveabook}. The second part follows from Proposition~\ref{prop:strassmantoalpha}.
\end{proof}
\begin{remark}
Note that if $\St(f)=1$, then we can guarantee that Newton's methods converge quadratically. This contrasts with the situation for Descartes' rule of signs, in which $\V(f)=1$ does not guarantee fast convergence for Newton's method.
\end{remark}

In the $p$-adic case, we can also consider the Strassman count for a particular closed ball as follows:
\begin{equation}
    \St(f;x,p^{-s}):=\St\left(f(x+p^{s}T)\right) , 
\end{equation}
where $x\in\Zp$ and $s\in\bbZ$. 
Similarly to the real case, 
we notice that the zeros of $f(x+p^{s}T)$ in $\Zp$ are in one-to-one correspondence with the roots of $f$ inside $\overline{B}(x,p^{-s})=x+p^{s}\Zp$. 

In the real setting, Descartes' rule of signs is 
an important ingredient of subdivision-based algorithms for 
isolating the real roots of real univariate polynomials.
Such algorithms, they also have excellent practical performance.
We aim to show that the same is true in the $p$-adic setting.
First, we describe what do the Descartes' rule of signs and Strassman's count actually count. Second, we demonstrate how both approaches lead to algorithms for solving polynomials.

\subsection{Exact counting}

In general, Descartes' rule of signs and Strassman's do not count exactly the number of
roots, in $\bbR$ and $\Zp$, respectively; they overestimate.
However, in both cases, the overestimation is due
to the presence of complex roots (respectively, of $\bbC$ and $\Cp$) nearby.
In the case of Descarte's rule of signs,
we can interpret the overestimation in the number of roots
using 
the so-called Obbherskoff areas and lenses.

Let $\varrho\in\bbN$ and $I= (a, b)$ a bounded interval. The \emph{Obreshkoff disc} $\uObD_{\varrho}(I)$ is the disc given by
\begin{equation}
    \bbD\left(\frac{a+b}{2}\left(1+i\frac{1}{\tan\frac{\pi}{\varrho+2}}\right),\frac{a+b}{2}\frac{1}{\sin\frac{\pi}{\varrho+2}}\right),
\end{equation}
whose boundary passes through the extremes of $I$ and whose center has an angle of $2\varphi:=\frac{\pi}{\varrho+2}$ in the triangle it forms with $I$. The \emph{Obreshkoff disc} $\dObD_{\varrho}(I)$ of $I$ is the conjugate of $\uObD_{\varrho}(I)$, having its center below $I$ instead than above $I$.
The \emph{Obreshkoff area} is 
\begin{equation}
\ObA_{\varrho}(I) = \mathsf{interior}(\uObD_{\varrho}(I) \,\cup\, \dObD_{\varrho}(I)),
\end{equation}
and the
\emph{Obreshkoff lens} is 
\begin{equation}
\ObL_{\varrho}(I) = \mathsf{interior}(\uObD_{\varrho}(I) \,\cap\, \dObD_{\varrho}(I)).
\end{equation}
We shows the Obreshkoff disks, area and lense in Figure~\ref{fig:Obreshkoff}. Note that 
\[
\ObL_d(I) \subset \ObL_{d-1}(I) \subset \cdots \subset \ObL_0(I)
\]
and that
\[
\ObA_0(I) \subset \ObA_1(I) \subset \cdots \subset \ObA_d(I).
\]

\begin{figure}[ht]
  \centering
  \includegraphics[scale=0.4]{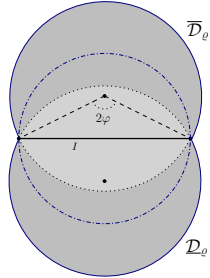}
  \caption{Obreshkoff discs, lens (light grey), and area  (light grey, grey) for an interval $I$.}
  \label{fig:Obreshkoff}
\end{figure}

The following theorem shows how the number of complex roots controls the Descartes count.

\begin{theo}[\cite{Obreshkoff-book}]
  \label{thm:Obr-signs}
  Let $f\in \bbR[T]$ be a real polynomial of degree $d$ and $I$
  a real open interval.
  If the Obreshkoff lens $\ObL_{d - k}(I)$ contains at least $k$ roots (counted with multiplicity) of $f$,
  then $k \leq V(f,I)$.
  If the Obreshkoff area $\ObA_{k}(I)$ contains at most $k$ roots (counted with multiplicity) of $f$,
  then $V(f,I) \leq k$.
  In particular,
  \begin{equation}
    \# \{ z\in\ObL_d(I)\mid f(x)=0 \}
    \le  V(f,I) \le
    \# \{ z\in \ObA_d(I)\mid f(x)=0 \},
  \end{equation}
  where the roots are counted with multiplicity.\eproof
\end{theo}

In the case of Strassman's counts, the theorem couldn't be more simple: it counts the number of roots in the closed unit ball of $\Cp$.

\begin{theo}\label{theo:strassmannexactforcomplex}
Let $f\in\Qp[T]$ be a $p$-adic polynomial, $x\in\Zp$ and $s\in\bbN$. Then
\begin{equation*}
    \St(f;x,p^{-s})=\#\{z\in\Cp\mid |z-x|\leq p^{-s},\,|z|\leq 1,\,f(z)=0\},
\end{equation*}
where the roots are counted with multiplicity.
\end{theo}
\begin{proof}
We only need to prove the claim for $s=0$ and $x=0$.
This is a consequence of the Newton polygon of $f$ counting the roots over $\Cp$, see~\cite[Theorem~6.4.7]{gouveabook}. To see the statement we only have to note that $\St(f)$ is the sum of the lengths of the non-positive slopes of the Newton polytope, and so the number of roots of $f$ in $\Cp$ with absolute value bounded by $1$.  
\end{proof}

\subsection{The subadditivity property of counting}

We want to use Descartes's rule of signs and Strassman's count for overcounting fast the number of roots in, respectively, intervals and closed balls. In the real case
we use the transformation in \eqref{eq:I-to-0inf} to count the number of roots
in an interval.
Another fundamental property of Descartes' count is the following one:

\begin{prop}\cite[Proposition 2.26]{eigenwillig_real_2010}
Let $f\in\bbR[T]$ be a real polynomial and $a_0,\ldots,a_n$ an strictly increasing sequence of real numbers. Then
\[
\sum_{i=0}^{n-1}\V(f;a_{i},a_{i+1})+\sum_{i=1}^{n-1}o(f,a_i)\}\leq \V(f;a_1,a_n),
\]
where $o(f,a_i)\in\bbN$ is the order of $f$ at $a_i$. In other words, Descartes' count (in an interval) is subadditive.
\end{prop}

We have an analogous statement
for the subadditivity property,
by substituting a union of disjoint intervals with an union of disjoint closed balls.

\begin{prop}\label{prop:strassmansubadditive}
Let $f\in\Qp[T]$ be a $p$-adic polynomial and also let
\[\overline{B}(x_1,p^{-s_1}),\ldots,\overline{B}(x_n,p^{-s_n})\]
be pairwise disjoint closed balls inside $\overline{B}(x,p^{-s})$. Then
\[
\sum_{i=0}^{r-1}\St(f;x_i,p^{-s_i})\leq \St(f;x,p^{-s}).
\]
In other words, Strassman's count (in a ball) is subadditive.
\end{prop}
\begin{proof}
By Theorem~\ref{theo:strassmannexactforcomplex}, $\St(f;y,p^{-t})$ counts the number of roots in $\Cp$ inside the closed ball
\[\overline{B}'(y,p^{-t}):=\{z\in\Cp\mid |z-y|\leq p^{-t}\}.\]
Now, if the $\overline{B}(x_i,p^{-s_i})$ are pairwise disjoint, then so are the $\overline{B}'(x_i,p^{-s_i})$; and if the $\overline{B}(x_i,p^{-s_i})$ are contain in $\overline{B}(x,p^{-s})$, so are the $\overline{B}'(x_i,p^{-s_i})$ in $\overline{B}'(x,p^{-s})$. Hence
\begin{multline}
    \sum_{i=0}^{r-1}\St(f;x_i,p^{-s_i})=\#\left\{z\in \bigcup_{i=1}^n\overline{B}'(x_i,p^{-s_i})\mid f(z)=0\right\}\\\leq \#\{z\in \overline{B}'(x,p^{-s})\mid f(z)=0\}=\St(f;x,p^{-s}),
\end{multline}
as we wanted to show.
\end{proof}

\subsection{Algorithms based on counting I: the real case}

We want to find the real roots of a real square-free polynomial $f\in\bbR[T]$ in an interval $I$. How can we do show? To do so, we will be subdividing the interval $I$ and `counting' the roots in each obtained interval $J$ until we can guarantee that every interval either contains no root of $f$ or contains a single root of $f$.

In the above process, one can use methods that produce an exact count such as Sturm sequences~\cite{dusharmayap2005}. However, exactness is not required for the counting method as long as we can guarantee that certain conditions are satisfied:
\begin{enumerate}
\item[(0)] The method is cheap to compute.
\item[(1)] The method does never undercount.
\item[(2)] The method is subadditive: the sum of the counts the method provides for subintervals $J_1,\ldots,J_s$ subdividing $I$ is at most the count it provides for $I$.
\item[(3)] If the method output zero or one, then the count is exact.
\item[(4)] If an interval is sufficiently small, then the method provides an exact count.
\end{enumerate}
Condition (0) justifies using an inexact count instead of an exact one, condition (1) allows us to terminate the algorithm at any point with the guarantee that we are bounding from above the number of roots; condition (2) guarantees that, at each subdivision, we cannot worsen our estimation; condition (3) allows us to easily terminate the algorithm at isolating intervals; and condition (4) guarantees that the algorithm will terminate at some point.

All the results until now shows that Descartes count, $\V$, satisfies these conditions. Because of this, one can use the Descartes count for isolating real roots as Algorithm \ref{alg:descartes}, \nameref{alg:descartes}, shows. We state the algorithm only for the interval $(-1,1)$ and assuming exact operations with real numbers to ease exposition. However, \nameref{alg:descartes} can be run with finite precision after some modification both in theory~\cite{sagraloffmehlhorn2016} and in practice~\cite{kobelrouilliersagraloff2016} (see also~\cite{eigenwillig_real_2010}).

\begin{algorithm*}[!htp]
\DontPrintSemicolon
\SetKwInOut{input}{Input}
\SetKwInOut{output}{Output}%\textsuperscript{a}
\caption{\textsc{Descartes}}\label{alg:descartes}
\input{$f\in \bbR[T]$}
\precondition{$f$ does not have singular roots in $(-1,1)$}
\nonl\hrulefill

$\mcS\leftarrow \{(-1,1)\}$\tcp*{Set of intervals to be processed}
$\mcZ\leftarrow \{\mcZ\}$\tcp*{Set of isolating intervals}
\tcc{Subdivision loop}
\Repeat{$\mcS=\varnothing$}{
Take $J=(a,b)\in\mcS$\;
$x_m\leftarrow \frac{a+b}{2}$
$J_l\leftarrow \left(a,x_m\right)$\;
$J_r\leftarrow \left(x_m,b\right)$\;
\tcc{Processing $x_m$}
\If{$f(x_m)=0$}{Add $\{x_m\}$ to $\mcZ$\tcp*{$x_m$ is a root of $f$}}
\tcc{Processing $J_l$}
$c_l\leftarrow \V(f,J_l)$\tcp*{We use Descartes count on $J_l$}
\uIf{$c_l>1$}{Add $J_l$ to $\mcS$\tcp*{Count too high, we have to subdivide $J_l$}}
\uElseIf{$c_l=1$}{Add $J_l$ to $\mcZ$\tcp*{$J_l$ contains exactly one root of $f$}}
\Else{Discard $J_l$\tcp*{$c_l=0$, and so $J_l$ does not contain any root.}}
\tcc{Processing $J_r$}
$c_r\leftarrow \V(f,J_r)$\tcp*{We use Descartes count on $J_r$}
\uIf{$c_r>1$}{Add $J_r$ to $\mcS$\tcp*{Count too high, we have to subdivide $J_r$}}
\uElseIf{$c_r=1$}{Add $J_r$ to $\mcZ$\tcp*{$J_r$ contains exactly one root of $f$}}
\Else{Discard $J_l$\tcp*{$c_r=0$, and so $J_r$ does not contain any root.}}
}
\tcc{Return of the isolating intervals}
\KwRet{$\mcZ$}

\nonl\hrulefill\\
\output{$J_1,\ldots,J_r\subseteq (-1,1)$}
\postcondition{The $J_i$ are pairwise disjoint\\
$\mcZ(f,(-1,1))\subset \bigcup_i J_i$\\
For all $i$, $\#\,J_i\cap \mcZ(f,(-1,1))=1$}
\end{algorithm*}

To analyze an algorithm such as \nameref{alg:descartes} is very convenient to consider the associated \emph{Descartes tree} $\mcT(f)$ obtained by the intervals $J$ appearing during the execution of \nameref{alg:descartes}$(f)$ and ordered by inclusion. The size of $\mcT(f)$ controls the run-time of \nameref{alg:descartes} at $f$. Now, to control the size of this tree, one normally separates the width and the height of this tree.

In general, the width of Descartes tree depends on the number of complex roots of $f$ nearby $(-1,1)$, and its depth on the separation of these complex roots, which can be controlled by the real condition number $f$ which is given by
\begin{equation}
    \kappa(f):=\sup_{x\in[-1,1]}\frac{\|f\|_1}{|f(x)|,|f'(x)|/d}\in [1,\infty],
\end{equation}
where $\|f\|_1:=\sum_{k}|f_k|$ is the $1$-norm of $f$. The following theorem summarizes the results,\footnote{The statement of Theorem~\ref{theo:descartescondbasedcomplexity} } in~\cite{TCTcubeI-journal} and in the recent~\cite{ETCT-descartes}, on the size of the Descartes tree---for full statement on complexity we refer to those papers.

\begin{theo}\label{theo:descartescondbasedcomplexity}
Let $f\in\bbR[T]$ be a real polynomial of degree $d$. Then:
\begin{enumerate}
    \item[(w)] The width of $\mcT(f)$ is at most
    \[
    \#\left\{\zeta\in\mcZ(f,\bbC)\mid \zeta\in\bigcup_{x\in [-1,1]}\bbD(x,(1-x)/4)\text{ or }\dist(z,[-1,1])\leq \frac{1}{d}\right\}.
    \]
    \item[(d)] The depth of $\mcT(f)$ is at most
    \begin{equation*}\tag*{\qed}
       6+\log \kappa(f)+\log d. 
    \end{equation*}
\end{enumerate}
\end{theo}

The importance of the above bound is that it can be used to obtain complexity bounds of \nameref{alg:descartes} for a random real polynomial $\fkf\in\bbR[T]$ of degree $d$. We state the result\footnote{The bound for the width of Theorem~\ref{theo:descartescondbasedcomplexity} can only be found in \cite{ETCT-descartes} for a random integer polynomial, but it can be easily generalized to the continuous case.} in a very specific case, even though the result holds in greater generality as it can be seen in~\cite{TCTcubeI-journal} and~\cite{ETCT-descartes}.

\begin{theo}\label{theo:descartesprobcomplexity}
Let $\fkf\in\bbR[T]$ be a random real polynomial of degree $d$ whose coefficients are independent random variables uniformly distributed in $[-1,1]$. Then:
\begin{enumerate}
    \item[(w)] The expected width of $\mcT(\fkf)$ is at most
    \[
    \Oh(\log^2 d).
    \]
    \item[(d)] The expected depth of $\mcT(\fkf)$ is at most
    \begin{equation*}\tag*{\qed}
     \Oh(\log d). 
    \end{equation*}
\end{enumerate}
\end{theo}

Combining these results with cost of arithmetic operations, we can obtain precise complexity bounds for \nameref{alg:descartes}. Our objective is to complete the analogy between the Descartes count and the Strassman count, by giving an algorithm in the $p$-adic case that uses Strassman count with a similar complexity analysis.

\subsection[Algorithms based on counting II: the p-adic case]{Algorithms based on counting II: the $p$-adic case}

Over the $p$-adic numbers, subdivisions are a lot nicer due to the metric (and topological structure) which allows us to subdivide $\Zp$ into into pairwise disjoint balls that are closed and open at the same time. However, the subdivision step of a closed ball $\overline{B}(x,p^{-s})$ into the $p$ closed balls
\[\overline{B}(x,p^{-(s+1)}),\overline{B}(x+2p^{-s},p^{-(s+1)}),\ldots,\overline{B}(x+(p-1)p^{-s},p^{-(s+1)})\]
can be problematic for big primes. Fortunately for us, we can handle this using fast factorization over $\bbF_p$.

We can now use Strassman count in order to provide an algorithm for finding roots as we did with Descartes count to provide one in the real case. We give this algorithm, \nameref{alg:strassman}, in  Algorithm~\ref{alg:strassman}. We postpone the full complexity and precision analysis to Section~\ref{sec:complexity}.

\begin{algorithm*}[!htp]
\DontPrintSemicolon
\SetKwInOut{input}{Input}
\SetKwInOut{output}{Output}%\textsuperscript{a}
\caption{\textsc{Strassman}}\label{alg:strassman}
\input{$f\in \Qp[T]$}
\precondition{$f$ does not have singular roots in $\Zp$}
\nonl\hrulefill

\tcc{Initial preparation}
$\ell_{\mathrm{in}}\leftarrow \St(f;0,1)$\;
$f_{\mathrm{in}}\leftarrow \sum_{k=0}^{\ell_{\mathrm{in}}} \frac{f_k}{f_{\ell_{\mathrm{in}}}}T^k$\tcp*{Normalization and truncation}
$\mcS\leftarrow \{(f_{\mathrm{in}};0,1;\ell_{\mathrm{in}})\}$\tcp*{Balls to subdivide}
$\mcZ\leftarrow \varnothing$\tcp*{Approximations to $p$-adic roots}
\tcc{Subdivision loop}
\Repeat{$\mcS=\varnothing$}{
Take $(g;x,p^{-s};\ell)\in\mcS$\;
\tcc{Subdivision step}
Find all $a_1,\ldots,a_l\in\{0,\ldots,p-1\}$ such that $|g(a_i)|<1$\tcp*{This means to solve $g\text{ (mod }p\text{)}$ since $\|g\|=1$ in our data structure}
\For{$i\in\{1,\ldots,l\}$}{
$x'\leftarrow x+a_ip^s$\tcp*{Update of the center of the ball}
$h\leftarrow f(x'+p^{s+1}T)$ truncated at degree $\ell$\tcp*{Note that $\St(h)=\St(f;x',p^{-(s+1)})$}
$\ell'\leftarrow \St(h)$\tcp*{Strassman count for $\overline{B}(x',p^{-(s+1)})$}
\uIf{$\ell'>1$}{
$\tilde{h}\leftarrow \sum_{k=0}^{\ell'} \frac{h_k}{h_{\ell'}}T^k$\tcp*{Normalization and truncation}
Add $\left(\tilde{h};x',p^{-(s+1)};\ell'\right)$ to $\mcS$\tcp*{Count too high, we need to subdivide}
}
\uElseIf{$\ell'=1$}{
Add $(x',p^{-(s+1)})$ to $\mcZ$\tcp*{Isolating ball found}
}
\Else{
\tcp{No root of $f$ in $\overline{B}(x',p^{-(s+1}))$}
}
}
}
\tcc{Return of the approximations}
\KwRet{$\mcZ$}

\nonl\hrulefill\\
\output{$(x_1,p^{-s_1}),\ldots,(x_r,p^{-s_r})\in \Zp\times p^{-\bbN}$}
\postcondition{$x_i\in [0,p^{s_i}-1]\cap\bbZ$\\
$\mcZ(f,\Zp)\subset \bigcup_i \overline{B}(x_i,p^{-s_i})$\\
For all $i$, $\#\,\overline{B}(x_i,p^{-s_i})\cap \mcZ(f,\Zp)=1$\\
the Newton iteration for $f$ starting at $x_i$ converges quadratically\\\hfill to the zero of $f$ in $\overline{B}(x_i,p^{-s_i})$}
\end{algorithm*}

To analyze \nameref{alg:strassman} we will follow the same path as we did with \nameref{alg:descartes}. In this way, we define the \emph{Strassman tree} $\mcT_p(f)$ as the tree whose vertices are the $(x,p^{-s})$ of the $(g;x,p^{-s};\ell)$ that belong to $\mcS$ during the computation of \nameref{alg:strassman}$(f)$ and are ordered by the inclusion of the $\overline{B}(x,p^{-s})$. Our objective is to bound not only the width and depth of this tree, but also the precision need for this algorithm to run correctly.

\begin{remark}
We note that there is a step of normalization and truncation of the polynomials. We observe that since $\St(f;x,p^{-s})$ decreases as we subdivide (Proposition~\ref{prop:strassmansubadditive}), we can just truncate the polynomials to that degree to save computation.
\end{remark}

\begin{remark}\label{rem:line7}
Line 7 of \nameref{alg:strassman} is the more problematic one. If $p$ is small, we can just go through the full $\bbF_p$ and do brute force. This would have a run-time of $\Oh(dp)$. If $p$ is large, this is not feasible. In that case, we compute $\gcd(g,x^p-x)\text{(mod }p\text{)}$, which can be done with run-time $\Oh(d^2\log p)$, and then we apply the Cantor-Zassenhaus factorization algorithm~\cite{cantorzassenhaus1981} which will take an average run-time of $\Oh(d^2\log^3d\log p)$. 

Hence, we have that line 7 can be done either in deterministic $\Oh(dp)$-time or in average $\Oh(d^2\log^3d\log p)$-time.
\end{remark}

\section{Condition numbers, separation and precision}\label{sec:condition}

In this section, we introduce the norm that we will working with $p$-adic polynomials. Using this norm, we define condition numbers, following a recipe analogous to that in~\cite{TCTcubeI-journal}, and show how it relates to the separation of the roots, Strassman count and the convergence of Newton's method---Hensel's lifting.

\subsection{Norms on polynomials}

Given a $p$-adic univariate polynomial $f=\sum_{k=0}^df_kT^k\in\Qp[T]$, we consider the following ultranorm
\begin{equation}
    \|f\|:=\max_{k}|f_k|.
\end{equation}
Using a norm we can quantify the perturbation of a polynomial.
The following proposition gives the main properties of the defined norm.

\begin{prop}\label{prop:norminequalities}
Let $f\in\Qp[T]$ be a $p$-adic polynomial. Then the following holds:
\begin{enumerate}
    \item[(e)] For every $k\in\bbN$ and $x\in\Zp$,
    \[|f^{(k)}(x)/k!|\leq \|f\|.\]
    In particular, $|f(x)|\leq \|f\|$ and $|f'(x)|\leq \|f\|$.
    \item[(i)] For every $x\in \Zp$,
    \[
    \|f(x+T)\|=\|f\|.
    \]
    In other words, the ultranorm $\|~\|$ is invariant under changes of variables coming from translations by an element in $\Zp$.
\end{enumerate}
\end{prop}
\begin{proof}
(e) We have that
\[
|f^{(k)}(x)/k!|=\left|\sum_{l\geq k}\binom{l}{k}f_lx^{l-k}\right|\leq \max_{l\geq k}\left|\binom{l}{k}\right||f_l||x|^{l-k}\leq \|f\|,
\]
since $\left|\binom{l}{k}\right|\leq 1$ and $|x|\leq 1$.

(i) Note that the coeffcients of $f(x+T)$ are precisely the $f^{(k)}(x)/k!$ due to Taylor's theorem. Hence, by (e), $\|f(x+T)\|\leq\|f\|$. Now, $f$ is obtained from $f(x+T)$ by doing a translation by $-x$. Therefore, by the same argument, $\|f\|\leq\|f(x+T)\|$, obtaining the desired equality.
\end{proof}

The following proposition will be useful later on. It shows that the norm controls the Lipschitz property of the derivative of a polynomial.

\begin{prop}\label{prop:Lipschitzpolynomial}
Let $f\in\Qp[T]$ be a $p$-adic polynomial. Then for all $x,y\in\Zp$,
\[
\left|\frac{|f^{(k)}(y)/k!|}{\|f\|}\right|\leq \max\left\{\left|\frac{|f^{(k)}(y)/k!|}{\|f\|}\right|,|x-y|\right\}.
\]
The previous relation holds with 
equality if $|x-y|<\left|\frac{|f^{(k)}(y)/k!|}{\|f\|}\right|$.
\end{prop}
\begin{proof}
Without loss of generality, we can assume that $\|f\|=1$. By Taylor's expansion, $f^{(k)}(y)/k!-f^{(k)}(x)/k!=\sum_{l\geq 1}\binom{k+l}{k}\left(f^{(k+l)}(x)/(k+l)!\right)(y-x)^{l}$. Thus, taking absolute values, applying the ultrametric inequality and Proposition~\ref{prop:norminequalities}, we obtain
\[
|f^{(k)}(y)/k!-f^{(k)}(x)/k!|\leq |x-y|.
\]
Thus $|f^{(k)}(y)/k!|\leq \max\{|f^{(k)}(x)/k!|, |x-y|\}$. The equality case follows from exchanging $x$ and $y$ under the given assumption.
\end{proof}

\subsection{Condition numbers and their properties}

We define the condition number over the $p$-adics following the definition in~\cite{CKMW1} for the complex case. 

\begin{defi}
Let $f\in\Qp[T]$. The \emph{local condition number of $f$ at $x\in\Zp$} is 
\begin{equation}
    \kappa(f,x):=\frac{\|f\|}{\max\{|f(x)|,|f'(x)|\}}\in (0,\infty].
\end{equation}
The \emph{global condition number of $f$} is 
\begin{equation}
    \kappa(f):=\sup_{z\in\Zp}\kappa(f,z)\in (0,\infty].
\end{equation}
\end{defi}

Note that $\kappa(f,x)$ is infinity if and only if $x$ is a singular root of $f$. Thus $\kappa(f)$ is finite as long as $f$ does not have singular roots in $\Zp$. Intuitively, the bigger $\kappa(f)$ is, the nearer $f$ is of having a singular zero in $\Zp$. The following theorem 
quantifies this statement and summarizes the main properties of $\kappa$---following the terminology introduced in~\cite{tonellicuetothesis}. We can consider it as a so-called condition number theorem.

\begin{theo}\label{theo:conditionproperties}
Let $f\in \Qp[T]$ be a $p$-adic polynomial and $x\in\Zp$. Then the following holds:
\begin{enumerate}
    \item[(0)] \emph{Bounds}: $1\leq \kappa(f,x)\leq\kappa(f)$.
    \item[(1)] \emph{Regularity inequality}: Either $|f(x)|/\|f\|\geq 1/\kappa(f,x)$ or $\|f'(x)\|/\|f\|\geq 1/\kappa(f,x)$.
    \item[(2)] \emph{1st Lipschitz property}: For every $g\in \Qp[T]$,
    \[
    \frac{\|g\|}{\kappa(g,x)}\leq \max\left\{\frac{\|f\|}{\kappa(f,x)},\|g-f\|\right\},
    \]
    with equality if $\|g-f\|/\|f\|<1/\kappa(f,x)$; and
    \[
    \frac{\|g\|}{\kappa(g)}\leq \max\left\{\frac{\|f\|}{\kappa(f)},\|g-f\|\right\},
    \]
    with equality if $\|g-f\|/\|f\|<1/\kappa(f)$.
    \item[(3)] \emph{2nd Lipschitz property}: For every $y\in\Zp$,
    \[
    \frac{1}{\kappa(f,y)}\leq \max\left\{\frac{1}{\kappa(f,x)},|y-x|\right\},
    \]
    with equality if $\kappa(f,x)|x-y|<1$.
    \item[(4)] \emph{Condition number theorem}: Let
    \[
    \Sigma_x:=\{g\in\Qp[T]\mid g(x)=g'(x)=0\}~\text{ and }~\Sigma:=\bigcup_{z\in\Zp}\Sigma_z
    \]
    be the set of $p$-adic polynomials with a multiple root at $x$
    and the set of $p$-adic polynomial with (at least one) multiple root in $\Zp$, respectively. Then
    \[
    \kappa(f,x)=\frac{\|f\|}{\dist(f,\Sigma_x)}~\text{ and }\kappa(f)=\frac{\|f\|}{\dist(f,\Sigma)}.
    \]
    \item[(5)] \emph{Higher derivative estimate}: If $\kappa(f,x)\|f'(x)\|/\|f\|\geq 1$, then
    \begin{equation}
        \gamma(f,x)\leq \kappa(f,x) ,
    \end{equation}
    where
    \[
    \gamma(f,x):=\begin{cases}\max_{k\geq 2}\left|f'(x)^{-1}f^{(k)}(x)/k!\right|^{\frac{1}{k-1}},&\text{if }f'(x)\neq 0\\\infty,&\text{otherwise}\end{cases}
    \]
    is \emph{Smale's $\gamma$ of $f$ at $x$}.
\end{enumerate}
\end{theo}
\begin{proof}
(0) This follows from Proposition~\ref{prop:norminequalities}.

(1) This is immediate from the definition of $\kappa$.

(2) We only prove the claim for the local condition number. For the global condition number, the claim follows by minimizing over $x\in\Zp$. We have that
\begin{align*}
    \|g\|/\kappa(g,x)&=\max\{|g(x)|,|g'(x)|\}\\
    &= \max\{|f(x)+(g-f)(x)|,|f'(x)+(g-f)'(x)|\}\\
    &\leq \max\{|f(x)|,|(g-f)(x)|,|f'(x)|,|(g-f)'(x)|\}&\text{(Ultrametric inequality)}\\
    &\leq \max\{|f(x)|,\|g-f\|,|f'(x)|,\|g-f\|\}&\text{(Proposition~\ref{prop:norminequalities}}\\
    &=\max\{\|f\|/\kappa(f,x),\|g-f\|\}.
\end{align*}
For the equality case, note that if $\|g-f\|/\|f\|<1/\kappa(f,x)$, then $\|g\|/\kappa(g,x)\leq \|f\|/\kappa(f,x)$ and, also, by symmetry,
\[
\|f\|/\kappa(f,x)\leq \max\{\|g\|/\kappa(g,x),\|g-f\|\}=\|g\|/\kappa(g,x),
\]
where the last equality follows from the fact that $\|g-f\|$ is smaller than $\|f\|/\kappa(f,x)$, so it cannot be the maximum bounding $1/\kappa(f,x)$. Thus $\|g\|/\kappa(g,x)= \|f\|/\kappa(f,x)$.

(3) Without loss of generality, we can assume that $\|f\|=1$ by scaling $f$ by an appropiate power of $p$. We have that
\begin{align*}
    1/\kappa(f,y)&=\max\{|f(y)|,|f'(y)|\}&\\
    &\leq \max\{|f(x)|,|f'(x)|,|y-x|\} &\text{(Proposition~\ref{prop:Lipschitzpolynomial})}\\
    &=\max\{1/\kappa(f,x),|y-x|\} .
\end{align*}
To prove the equality, we interchange $x$ and $y$ and argue,
mutatis mutandis, as in (2).

(4) We only prove the local version. The global version follows from the global one by minimizing over all $x$. By the 1st Lipschitz property, for every $g\in \Sigma_x$,
\[
\|f\|/\kappa(f,x)\leq \max\{0,\|f-g\|\}=\|f-g\|,
\]
since $\|g\|/\kappa(g,x)=0$. Thus $\kappa(f,x)\geq \|f\|/\dist(f,\Sigma)$.

To prove the other inequality, note that
$f-f(x)-f'(x)T\in\Sigma_x$.
Thus
\[
\dist(f,\Sigma)\leq \|f(x)+f'(x)T\|=\max\{|f(x)|,|f'(x)|\}=\|f\|/\kappa(f,x).
\]
Hence $\kappa(f,x)\leq \|f\|/\dist(f,\Sigma_x)$.

(5) Under the given assumption, the regularity inequality implies $\|f\|/|f'(x)|\leq \kappa(f,x)$. Hence
\[
\gamma(f,x)=\max_{k\geq 2}\left(\frac{|f^{(k)}(x)/k!|}{|f'(x)|}\right)^{\frac{1}{k-1}}\leq \max_{k\geq 2}\left(\frac{\|f\|}{|f'(x)|}\right)^{\frac{1}{k-1}}\leq \max_{k\geq 2}\kappa(f,x)^{\frac{1}{k-1}},
\]
where the first inequality follows from Proposition~\ref{prop:norminequalities}. Now, $\kappa(f,x)\geq 1$, so the right-hand side is bounded by $\kappa(f,x)$,
that concludes the proof.
\end{proof}

The following proposition relates the local condition number to Strassman's count. We will give in the next subsection an alternative proof which uses the condition-based separation bounds to prove the stated result.

\begin{prop}\label{prop:conditionminimum}
Let $f\in\Qp[T]$ be a $p$-adic polynomial and $x\in\Zp$. For all
\begin{equation}
    s\geq 1+\frac{\log\kappa(f,x)}{\log p},
\end{equation}
it holds $\St(f;x,p^{-s})\leq 1$.
\end{prop}
\begin{proof}
By the choice of $s$, we have that $\|f\|\leq p^{s-1}\max\{|f(x)|,|f'(x)|\}$. Therefore for all $k\geq 2$,
\begin{multline*}
    |f^{(k)}(x)/k!|/p^{ks}\leq \|f\|/p^{ks}\leq p^{-(k-1)s-1}\max\{|f(x)|,|f'(x)|\}\\<p^{-s}\max\{|f(x)|,|f'(x)|\}\leq \max\{|f(x)|,|f'(x)|/p^s\},
\end{multline*}
where the first inequality follows from Proposition~\ref{prop:norminequalities}, the second from the coice of $s$---see first sentences in this proof---, and the third one follows from $k\geq 2$.

Hence $\St(f;x,p^{-s})\leq 1$, since the absolute value of the coefficients of $f(x+p^sT)$ are $|f(x)|,|f'(x)|/p^s,\ldots,|f^{(k)}(x)/k!|/p^{ks},\ldots$, we conclude the proof.
\end{proof}

\subsection[Smale's alpha theory, Newton's method and Hensel's lemma]{Smale's $\alpha$-theory, Newton's method and Hensel's lemma}

Smale's $\alpha$-theory gives sufficient conditions for the convergence of the Newton's method. In the $p$-adic univariate setting, Smale's $\alpha$-theory---reproducing the proofs with the ultrametric inequality---reduces itself to the famous Hensel's lemma~\cite[pp.~70--72]{gouveabook}. In the multivariate setting, this gives better criteria than the criteria based on the Jacobian~\cite{conradhensel}.
In our understanding, this version of Smale's $\alpha$-theory is unknown in the $p$-adic setting; thus we present it in detail 
in the Appendix~\ref{sec:smalealphatheory}.

To define Smale's $\alpha$-theory in the univariate $p$-adic setting we need to introduce the Smale's parameters as follows:
\begin{defi}
Let $f\in\Qp[T]$ and $x\in\Qp$. Then, we define the following quantities:
\begin{enumerate}[(a)]
    \item \emph{Smale's $\alpha$}: $\alpha(f,x):=\beta(f,x)\gamma(f,x)$, if $f'(x)\neq 0$, and $\alpha(f,x):=\infty$, otherwise.
    \item \emph{Smale's $\beta$}: $\beta(f,x):=|f(x)/f'(x)|$, if $f'(x)\neq 0$, and $\alpha(f,x):=\infty$, otherwise.
    \item \emph{Smale's $\gamma$}: $\gamma(f,x):=\max_{k\geq 2}\left|\frac{f^{(k)}(x)}{k!f'(x)}\right|^{\frac{1}{k-1}}$, if $f'(x)\neq 0$, and $\gamma(f,x):=\infty$, otherwise.
\end{enumerate}
\end{defi}
\begin{remark}
For a root $\zeta$ of $f$, we notice that
\[-\frac{\log\gamma(f,\zeta)}{\log p}\]
is the first slope of the Newton polygon of $(T-\zeta)f(T)$. This provides a nice interpretation of Smale's $\gamma$ in the $p$-adic univariate case. Moreover, as we will show in Theorem~\ref{theo:gammaseparationtheorem}, it has also a geometric relation to the separation of the complex $p$-adic roots of a polynomial.
\end{remark}

We can consider the \emph{Newton operator}
\[\newton_f:x\mapsto x-f(x)/f'(x) ,\]
at those points where $f'$ is non-zero. In general, if we do not choose the point $x$ carefully, the Newton operator does not converge ---using Hensel's lemma terminology, the approximate root does not lift. So the question is: under which conditions can we guarantee that the sequence 
\[
x,\newton_f(x),\newton_f^2(x),\ldots,
\]
is well-defined and converges to a root of $f$ fast? Smale's $\alpha$-theorem gives sufficient conditions for this convergence to happen. Moreover, it gives conditions under which the convergence is quadratic ---the number of exact digits doubles at each iteration. Smale's $\gamma$-theorem gives the same guarantees for points sufficiently close to a non-singular root. In the $p$-adic setting, we can unify these two theorems as follows:

\begin{theo}[$p$-adic Smale's $\alpha$/$\gamma$-theorem]\label{theo:padicSmalealphatheory}
Let $f\in\Cp[T]$ and $x\in\Cp$. Then the following are equivalent:
\begin{itemize}
    \item[($\alpha$)] ($\alpha$-criterion) $\alpha(f,x)<1$.
    \item[($\gamma$)] ($\gamma$-criterion) $\dist(x,f^{-1}(0))<1/\gamma(f,x)$.
\end{itemize}
Moreover, if any of the above (equivalent) conditions holds, then the Newton sequence, $\{\newton_f^k(x)\}$, is well-defined and it converges quadratically to a non-singular root $\zeta$ of $f$. 
In particular, for all $k$, the following holds:
\begin{enumerate}
    \item[(a)] $\alpha(f,\newton_f^k(x))\leq \alpha(f,x)^{2^{k}}$.
    \item[(b)] $\beta(f,\newton_f^k(x))\leq \beta(f,x)\alpha(f,x)^{2^{k}}$.
    \item[(c)] $\gamma(f,\newton_f^k(x))\leq \gamma(f,x)$.
    \item[(Q)] $|\newton_f^k(x)-\zeta|=\beta(f,\newton_f^k(x))\leq \alpha(f,x)^{2^{k}}\beta(f,x)<\alpha(f,x)^{2^{k}}/\gamma(f,x)$.
\end{enumerate}
\end{theo}
\begin{proof}
See the Appendix~\ref{sec:smalealphatheory} for the proof of the statement in full generality.
\end{proof}
\begin{remark}
Note that $\beta(f,x)=|\newton_f(x)-x|$. Thus $\beta(f,x)$ is nothing more than the length of a Newton step. In other words, Smale's $\alpha$-theorem tells us that if the Newton step is sufficiently small, then fast convergence is guaranteed. As we will see in the sequel (Proposition~\ref{prop:conditiontoalpha}), $\alpha(f,x)<1$ is implied by the condition
\[
\|f\||f(x)|<|f'(x)|^2,
\]
which, when $f\in\Zp[T]$, it gives the strong version of Hensel's lemma for lifting roots.
\end{remark}
\begin{remark}
We also note that we are stating the result for complex $p$-adics. However, if the considered polynomial and initial point lie in $\Qp$, then we can guarantee that 
\end{remark}

The following propositions relates $\alpha$ to the condition number and to Strassman's count.

\begin{prop}\label{prop:conditiontoalpha}
Let $f\in \Qp[T]$ and $x\in\Zp$. If $\kappa(f,x)|f(x)|/|f'(x)|<1$, then $\alpha(f,x)<1$.
\end{prop}
\begin{proof}
We use Theorem~\ref{theo:conditionproperties}: the regularity inequality---since $\kappa(f,x)|f(x)| / |f'(x)|<1$ implies $\kappa(f,x)|f(x)|/\|f\|<1$ by Proposition~\ref{prop:norminequalities}---and the higher derivative estimate to bound $\gamma(f,x)$.
\end{proof}
\begin{prop}\label{prop:strassmantoalpha}
Let $f\in \Qp[T]$, $x\in\Zp$ and $s\in\bbN$. If $\St(f;x,p^{-s})=1$, then
\begin{enumerate}
    \item[(S1)] $\alpha(f,x)<1$, $\beta(f,x)\leq p^{-s}$, and $\gamma(f,x)< p^s$.
    \item[(S2)] The Newton sequence staring at $x$, $\{\newton_f^k(x)\}$, is well-defined and converges to the only root $\zeta$ of $f$ in $\overline{B}(x,p^{-s})$---and in the corresponding closed ball in $\Cp$.
    \item[(S3)] For all $k$, $|\newton_f^k(x)-\zeta|\leq p^{-s-1-2^{k-1}}$.
\end{enumerate}
\end{prop}
\begin{proof}
If $\St(f;x,p^{-s})=1$, then $|f(x)|\leq |f'(x)|/p^{s}$ and for $k\geq 2$,
\[
|f^{(k)}(x)/k!|<|f'(x)|p^{(k-1)s}.
\]
Therefore $\beta(f,x)\leq p^{-s}$ and $\gamma(f,x)<p^s$. Thus $\alpha(f,x)<1$ and the rest follows from Theorem~\ref{theo:padicSmalealphatheory}. 

Note that in $\Qp$, we have that $|x|<1$ implies $|x|\leq 1/p$, so we have that $\gamma(f,x)\leq p^{s-1/(d-1)}$ and so $\alpha(f,x)\leq p^{-1/(d-1)}$. Then, a direct application of Theorem~\ref{theo:padicSmalealphatheory} gives $|\newton_f^k(x)-\zeta|\leq p^{-s-2^k/(d-1)}$. Now, by the same theorem,
\[\beta(f,\newton_f(x))\leq \beta(f,x)\alpha(f,x)=\beta(f,x)p^{-1/(d-1)}~\text{ and }~\gamma(f,\newton_f(x))=\gamma(f,x).\]
But $\beta(f,\newton_f(x))$ is the absolute value of an element of $\Qp$, so $\beta(f,\newton_f(x))\leq p^{-s-1}$ and so $\alpha(f,\newton_f(x))\leq 1/p$. Hence, applying Theorem~\ref{theo:padicSmalealphatheory} to $\newton_f(x)$, we have the desired conclusion.
\end{proof}
\begin{remark}
The fact that $\St(f;x,p^{-s})=1$ implies quadratic convergence of the Newton's method means that the approximations obtained by Strassman are better than those obtained by Descartes. In the latter, there are no guarantees that the Newton method starting at the extremes of the isolating intervals converges at all, while, at the roots isolated using Strassman, Newton's method does not only converge, but it does so quadratically.
\end{remark}

\subsection{Separation bounds}

How separated are the roots of a $p$-adic polynomial? Smale's $\alpha$-theory provides only a bound in the real case~\cite{dedieubook}, although it is not exact. In the $p$-adic case, we have a precise estimate of the separation. Notice that  the separation bounds measures the separation (i.e., the distance) between the complex $p$-adic roots. 
	
\begin{defi}
Let $f\in\Cp[T]$ and $\zeta\in\Cp$ a root of $f$. The \emph{local separation of $f$ at $\zeta$} is
\begin{equation}
    \Delta_{\zeta}(f):=\min\{|z-\zeta|\mid z\in\Cp,\,z\neq \zeta,\,f(z)=0\}=\frac{1}{\gamma(f,\zeta)}, 
\end{equation}
if $\zeta$ is non-singular, and $0$ otherwise. The \emph{separation of $f$} is
\begin{equation}
    \Delta(f):=\min\{\Delta_{\eta}(f)\mid \eta\in\Cp,\,f(\eta)=0\}.
\end{equation}
\end{defi}

\begin{theo}\label{theo:gammaseparationtheorem}
Let $f\in\Cp[T]$ and $\zeta\in\Cp$ a root of $f$. Then
\[
\Delta_{\zeta}(f)=\frac{1}{\gamma(f,\zeta)}.
\]
\end{theo}
\begin{cor}\label{cor:kappaseparation}
Let $f\in\Qp[T]$ and $\zeta\in\Zp$ a root of $f$. Then
\[
\Delta_{\zeta}(f)\geq \frac{1}{\kappa(f,\zeta)}\geq \frac{1}{\kappa(f)}.
\]
\end{cor}
\begin{remark}
Unlike the real case~\cite{TCTcubeI-journal}, note that the bounds above do not depend on the degree of the polynomial considered.
\end{remark}
\begin{proof}[Proof of Theorem~\ref{theo:gammaseparationtheorem}]
By Theorem~\ref{theo:padicSmalealphatheory}, we have that $\Delta_{\zeta}(f)\geq\frac{1}{\gamma(f,\zeta)}$. Without loss of generality, assume that $\zeta=0$ and that it is non-singular. Then, we can write $f$ as
\[
T\sum_{k=1}^{d}f_kT^{k-1},
\]
where $f_1\neq 0$. By definition, $\Delta_0(f)$ is the absolute value of the smallest root of $\sum_{k=1}^{d}f_kT^{k-1}$. Now, by~\cite[6.4.7]{gouveabook}, this can be computed by finding the smallest slope of the Newton polygon of $\sum_{k=1}^{d}f_kT^{k-1}$, which is the first possible slope. Now, the possible first slope are
\[
\frac{\nu(f_k)-\nu(f_1)}{k-1}~~(k\geq 2) ,
\]
where $\nu:\Cp\rightarrow \bbR$ is the valuation of $\Cp$. Now, taking the minimim of these, we obtain the first slope, and so the smallest root of $\sum_{k=1}^{d}f_kT^{k-1}$ has norm
\[
p^{\min_{k\geq 2}\frac{\nu(f_k)-\nu(f_1)}{k-1}}=\min_{k\geq 2}\left|\frac{f_1}{f_k}\right|^{\frac{1}{k-1}}=\frac{1}{\gamma(f,x)},
\]
as we wanted to show.
\end{proof}
\begin{proof}[Proof of Corollary~\ref{cor:kappaseparation}]
This is Theorem~\ref{theo:gammaseparationtheorem} combined with the higher derivative estimate (Theorem~\ref{theo:conditionproperties}).
\end{proof}

We can now provide an alternative proof of Proposition~\ref{prop:conditionminimum} using the separation of the roots.

\begin{proof}[Alternative proof of Proposition~\ref{prop:conditionminimum}]
If $\St(f;x,p^{-s})=0$, we are done. If $\St(f;x,p^{-s})\geq 1$, take a root $\zeta\in\Cp$ of $f$ such that $|\zeta-x|\leq p^{-s}$---and so $|\zeta|\leq 1$. By the choice of $s$, this means that
\[
|\zeta-x|\leq \frac{1}{p\kappa(f,x)}<\frac{1}{\kappa(f,x)},
\]
and so, by the 2nd Lipschitz property~\ref{theo:conditionproperties}---we only need $|x|,|y|\leq 1$, not $x,y\in\Zp$---,
\[
\kappa(f,x)=\kappa(f,\zeta).
\]
But this means, by Corollary~\ref{cor:kappaseparation}, that for any other root $\eta\in\Cp$ of $f$, we have that
\[
|\eta-\zeta|\geq\frac{1}{\kappa(f,x)}>p^{-s}.
\]
Hence, by Theorem~\ref{theo:strassmannexactforcomplex}, $\St(f;x,p^{-s})=1$.
\end{proof}

\subsection{Precision}

The following theorem shows how condition numbers allow us to truncate the coefficients of a polynomial so that the roots of the approximation are roots of the original polynomial \`{a} la Smale---meaning that the Newton method starting at these roots converge quadratically to the roots of the original polynomial.

\begin{theo}\label{theo:preccondnumber}
Let $f,\tilde{f}\in\Qp[T]$ be a $p$-adic polynomials of degree $d$. If
\[\kappa(f)^2\frac{\left\|\tilde{f}-f\right\|}{\|f\|}<1,\]
then:
\begin{enumerate}[(i)]
    \item $f$ and $\tilde{f}$ have the same number of roots in $\Zp$.
    \item For every root $\tilde{\zeta}\in\Zp$ of $\tilde{f}$, $\alpha(f,x)\leq 1/p$. In particular, there is a unique root $\zeta\in\Zp$ of $f$ such that for $k\geq 0$,
    \[\left|\newton_f^k\left(\tilde{\zeta}\right)-\zeta\right|\leq p^{-2^k}.\]
    \item For every root $\zeta\in\Zp$ of $f$, $\alpha(\tilde{f},x)\leq 1/p$. In particular, there is a unique root $\tilde{\zeta}\in\Zp$ of $\tilde{f}$ such that for $k\geq 0$,
    \[\left|\newton_{\tilde{f}}^k\left(\zeta\right)-\tilde{\zeta}\right|\leq p^{-2^k}.\]
\end{enumerate}
\end{theo}
\begin{exam}
Consider the $2$-adic polynomial $f=2x+x^2$. For this polynomial, we have
\[
\kappa(f,0)=2,\,\kappa(f,1)=1,\,\kappa(f,2)=2
\]
and so $\kappa(f)=2$. Moreover, note that $\|f\|=1$ and that $f$ has $2$ roots in $\bbZ_2$: $0$ and $-2$.

Now consider, $\tilde{f}=4+2x+x^2$. Even though
\[\kappa(f)\left\|\tilde{f}-f\right\|/\|f\|<1,\]
and so $\kappa(\tilde{f})=2$ and $\left\|\tilde{f}\right\|=1$, we have that $f$ does not have any root in $\bbQ_2$, because its discriminant, $-12$, is not an square in $\bbQ_2$, because $-3$ is not one modulo $8$. Hence the square in the condition number of the condition of Theorem~\ref{theo:preccondnumber} cannot be removed in general.
\end{exam}

\begin{proof}[Proof of Theorem~\ref{theo:preccondnumber}]
If $\kappa(f)\frac{\left\|\right\|}{\|f\|}<1$, then $\left\|\tilde{f}-f\right\|<\|f\|$ and so $\|f\|=\left\|\tilde{f}\right\|$. Now, we can assume, without loss of generality, after scaling by a power of $p$, that $\|f\|=\left\|\tilde{f}\right\|=1$.

By the 1st Lipschitz property (Theorem~\ref{theo:conditionproperties}), we have that $\kappa(f)=\kappa(\tilde{f})$ and that for all $x\in\Zp$, $\kappa(f,x)=\kappa(\tilde{f},x)$. Thus, once we show (ii), we are done, since, on the one hand, we can interchange the roles of $f$ and $\tilde{f}$, so (ii) gives (iii); and, on the other hand, once we have (ii) and (iii), we have injective maps from the roots of $f$ in $\Zp$ to the roots of $\tilde{f}$ in $\Zp$ and in the other direction.

Let $\tilde{\zeta}$ be a root of $\tilde{f}$. Then, on the one hand,
\[
\left|f(\tilde{\zeta})\right|\leq \left\|\tilde{f}-f\right\|,
\]
and, on the other hand,
\[
\left|f'(\tilde{\zeta})\right|=\left|\tilde{f}'(\tilde{\zeta})\right|=1/\kappa(f,\tilde{\zeta})=1/\kappa(\tilde{f},\tilde{\zeta}),
\]
because
\begin{multline*}
 \left|f'(\tilde{\zeta})\right|=\max\left\{\left|\tilde{f}'(\tilde{\zeta})\right|,\left|\left(\tilde{f}-f\right)(x)\right|\right\}\\=\max\left\{1/\kappa\left(\tilde{f},\tilde{\zeta}\right),\left|\left(\tilde{f}-f\right)(x)\right|\right\}=1/\kappa\left(\tilde{f},\tilde{\zeta}\right)= \left|\tilde{f}'(\tilde{\zeta})\right|,  
\end{multline*}
where equailities follows from the equality case of the ultrametric inequality, $\left|\left(\tilde{f}-f\right)(x)\right|\leq\left\|\tilde{f}-f\right\|$, by Proposition~\ref{prop:norminequalities}; $\left|\tilde{f}'(\tilde{\zeta})\right|=1/\kappa\left(\tilde{f},\tilde{\zeta}\right)$, due to $\tilde{f}\left(\tilde{\zeta}\right)=0$; and our assumption. Therefore
\[
\beta\left(f,\tilde{\zeta}\right)\leq \kappa\left(f,\tilde{\zeta}\right)\left\|\tilde{f}-f\right\|,
\]
and, by the the higher derivative estimate (Theorem~\ref{theo:conditionproperties}), 
\[
\gamma\left(f,\tilde{\zeta}\right)\leq \kappa\left(f,\tilde{\zeta}\right).
\]
Thus
\[
\alpha(f,\tilde{\zeta})\leq \kappa(f,\tilde{\zeta})^2\left\|\tilde{f}-f\right\|<1,
\]
and Theorem~\ref{theo:padicSmalealphatheory} finishes the proof---note that $\kappa(f,\tilde{\zeta})^2\left\|\tilde{f}-f\right\|<1$ implies that it is at most $1/p$ since it is the product of norms of vectors with entries in $\Qp$.
\end{proof}

We note that the above bound might be problematic to use in practice due to the issue that to compute it we need to have already compute the condition number $\kappa(f)$, which is not necessarily an easy task.

\section{Probabilistic complexity analysis}\label{sec:random}

In this, a \emph{random $p$-adic polynomial $\fkf\in\Zp[T]$ of degree $d$} is a random $p$-adic polynomial
\[
\fkf=\sum_{k=0}^d \fkf_kT^K
\]
where the $\fkf_k$ are independent random $p$-adic variables uniformly distributed in $\Zp$ (with respect the Haar probability measure). In other words, we are taking the Haar measure the $\Zp$-module of polinomials of degree $d$ in $\Zp[T]$. 

We aim to prove probabilistic results for this class of random polynomials. First, we recall some basic facts on random $p$-adic vectors; second, we analyze probabilistically the condition number; third, we analyze probabilistically Strassman count; and fourth and last, we apply these results to the results in previous section to obtain the probabilistic analysis of \nameref{alg:strassman}.

\subsection[Random p-adic vectors and some basic results]{Random $p$-adic vectors and some basic results}

Since we will not be considering more than a class of random $p$-adic vectors. We can give the following definition for a random $p$-adic vector.

\begin{defi}
A \emph{random $p$-adic vector} $\fkx\in\Zp^N$ is a random element of $\Zp$ taken with respect the unique Haar measure of $\Zp^N$.
\end{defi}

The following proposition list the elementary facts that we will be using regarding a random $p$-adic vector. 

\begin{prop}\label{prop:randompadicvectors}
Let $\fkx\in\Zp^N$ be a random $p$-adic vector. Then:
\begin{enumerate}[(a)]
    \item For every $A\in \mathrm{GL}_N(\Zp)$ and $x\in \Zp$, $x+A\fkx$ is a random $p$-adic vector.
    \item For every $s\in\bbN$,
    \[\bbP(\|x\|\leq p^{-s})=p^{-Ns}~\text{ and }~\bbP(\|x\|=p^{-s})=\left(1-p^{-s}\right)p^{-Ns}.\]
    \item For every $i$, $(\fkx_1,\ldots,\fkx_i)$ and $(\fkx_{i+1},\ldots,\fkx_N)$ are independent random $p$-adic vectors.
    \item If $\fky\in\Zp^{M}$ is a random $p$-adic vector, then so it is $(\fkx,\fky)\in\Zp^{N+M}$.
\end{enumerate}
\end{prop}
\begin{proof}
(a) This follows from the fact that $B\mapsto \bbP(x+A\fkx\in B)$ defines a Haar measure on $\Zp^N$. So it has to agree with the Haar measure of $\Zp^N$.

(b) This follows from the fact that for a Haar measure all the closed balls of the same radious have the same measure and that there are $p^{Ns}$ closed balls of radious $p^{-s}$ in $\Zp^N$.

(c) and (d). This follows from the fact that the product of the Haar probability measures is the Haar probability measure of the product.
\end{proof}

When we apply this proposition to our random $p$-adic polynomial, we get the following proposition:

\begin{prop}\label{prop:randompadicpolynomial}
Let $\fkf\in\Zp[T]$ be a random $p$-adic polynomial of degree $d$. Then:
\begin{enumerate}[(a)]
\item For every $x\in\Zp$, $\fkf(x+T)$ is also a random $p$-apolynomial of degree $d$.
\item For every $x\in\Zp$, $(\fkf^{(k)}(x)/k!)_{i=0}^d\in\Zp^{d+1}$ is a random $p$-adic vector.
\item For $s\in\bbN$, $\bbP(\|\fkf\|\leq p^{-s})=p^{-s(d+1)}$. In particular, for all $k\geq 1$,
\[
\bbE\ln^k\frac{1}{\|\fkf\|}\leq k^{k}.
\]
\end{enumerate}
\end{prop}
\begin{proof}
(a) This follows from Proposition~\ref{prop:randompadicvectors} (a) and the fact that $f\mapsto f(x+T)$ is a $\mathrm{GL}$-transformation of the space of $p$-adic integer polynomials of degree $d$.

(b) Since $\fkf(x+T)$ is a random $p$-adic polynomial of degree $d$, the coefficients of $\fkf(x+T)$ form a random $p$-adic vector in $\Zp^{d+1}$, by definition of random $p$-adic polynomial.

(c) The first part follows from Proposition~\ref{prop:randompadicvectors} (b). For the second part, note that for all $s\geq 0$, not necessarily a natural number,
\[
\bbP(\|\fkf\|\leq p^{-s})=\bbP(\|\fkf\|\leq p^{-\lceil s\rceil})=p^{-\lceil s\rceil}\leq p^{-s}.
\]
Thus, for $s\geq 0$,
\[
\bbP\left(\ln\frac{1}{\|\fkf\|}\geq s\right)\leq \enumber^{-s},
\]
and so for $k\geq 1$,
\[
\bbE\ln^k\frac{1}{\|\fkf\|}=\int_{0}^\infty ku^{k-1}\bbP\left(\ln\frac{1}{\|\fkf\|}\geq u\right)\,\mathrm{d}u\leq \int_{0}^\infty ku^{k-1}e^{-u}\,\mathrm{d}u\leq \Gamma(k+1)\leq k^k,
\]
as claimed.
\end{proof}

\subsection{Probabilistic analysis of the condition number}

The analysis of the condition number leads us to the following:

\begin{theo}\label{theo:probcondition}
Let $\fkf\in\Zp[T]$ a random $p$-adic polynomial of degree $d$. Then, for every $x\in\Zp$ and $s\geq 0$,
\[
\bbP\left(\kappa(\fkf,x)\geq p^{s}\right)\leq p^{-2s},
\]
and, for every $s\geq 0$,
\[
\bbP\left(\kappa(\fkf)\geq p^{s}\right)\leq p^{-s}.
\]
\end{theo}
\begin{cor}\label{cor:probcondition}
Let $\fkf\in\Zp[T]$ a random $p$-adic polynomial of degree $d$. Then for all $k\geq 1$,
\[
\bbE\ln^k\kappa(\fkf)\leq k^k.
\]
\end{cor}
\begin{remark}
Note that if $d\geq 2$, then $\frac{1}{2}p^{-2s}\leq \bbP\left(\kappa(f,x)\geq p^{s}\right)$. So the bound in Theorem~\ref{theo:probcondition} for the local condition number is almost-optimal.
\end{remark}
\begin{remark}
Note that this shows that the Strassman tree's depth is not only very small with high probability, but it is of constant depth with high probability.
\end{remark}

The above theorem will follow from the following proposition.

\begin{prop}\label{prop:padicprobability}
Let $A:\Zp^N\rightarrow \Zp^r$ be a linear orthogonal projection, i.e., $A$ can be extended to a linear in $\mathrm{GL}_N(\Zp)$. Then, for all $s\in\bbZ_{>0}$,
    \[
    \bbP\left(\frac{\|x\|}{\|Ax\|}\geq p^{s}\right)=\frac{\left(1-p^{r-N}\right)}{\left(1-p^{-N}\right)}p^{-rs}\leq p^{-rs}.
    \]
\end{prop}

\begin{proof}[Proof of Theorem~\ref{theo:probcondition}]
The first part follows from Proposition~\ref{prop:padicprobability}, since the map
\[f\mapsto (f(x),f'(x))\in\Zp^2\]
is an orthogonal projection, since it can be extended to the $\mathrm{GL}_{d+1}$ map $f\mapsto f(x+T)$, obtained by performing a translation of the variable by $x\in\Zp$.

For the second part, if $\kappa(f)\geq p^s$, then for some $x_*\in\Zp$, $\kappa(f,x_*)\geq p^s$. Therefore, by 2nd Lipschitz property, for all $y\in \overline{B}(x,p^{-s})$, $\kappa(f,y)\geq p^s$. Hence $\kappa(f)\geq p^s$ implies
\[
\bbP_\fkx(\kappa(f,\fkx)\geq p^{s})\geq p^{-s},
\]
where $\fkx\in\Zp$ is a random $p$-adic. In this way,
\begin{align*}
    \bbP_\fkf(\kappa(\fkf)\geq p^s)&\leq \bbP_\fkf(\bbP_\fkx(\kappa(\fkf,\fkx)\geq p^{-s})\geq p^s)&\text{(Above discussion)}\\
    &\leq p^s\bbE_\fkf\bbP_\fkx(\kappa(\fkf,\fkx)\geq p^{-s})&\text{(Markov's inequality)}\\
    &=p^s\bbE_\fkx\bbP_\fkf(\kappa(\fkf,\fkx)\geq p^{-s})&\text{(Tonelli's theorem)}\\
    &\leq p^{-s}&\text{(First part)}
\end{align*}
Note that we can apply Tonelli's theorem, because the Haar measure of a product is the product of the Haar measures, $\bbP_\fkx(\kappa(\fkf,\fkx)\geq p^{-s})=\bbE_\fkx\chi_{\{x\mid\kappa(\fkf,x)\geq p^{-s}\}}$---$\chi$ is the indicator function---and $\bbP_\fkf(\kappa(\fkf,\fkx)\geq p^{-s})=\bbE_\fkf\chi_{\{f\mid\kappa(f,\fkx)\geq p^{-s}\}}$.
\end{proof}
\begin{proof}[Proof of Corollary~\ref{cor:probcondition}]
Since $\kappa(\fkf)\geq 1$, by Theorem~\ref{theo:conditionproperties} (0), we have that
\[
\bbE\ln^k\kappa(\fkf)=\int_1^\infty ku^{k-1}\bbP(\ln\kappa(\fkf)\geq u)\,\mathrm{d}u.
\]
Now, 
\begin{align*}
  \bbP(\ln\kappa(\fkf)\geq u)&=\bbP\left(\kappa(\fkf)\geq \enumber^u\right)\\
  &=\bbP\left(\kappa(\fkf)\geq p^{\frac{u}{\ln p}}\right)\\
  &=\bbP\left(\kappa(\fkf)\geq p^{\left\lceil\frac{u}{\ln p}\right\rceil}\right)&(\kappa(\fkf)\in p^{\bbN})\\
  &\leq p^{-\left\lceil\frac{u}{\ln p}\right\rceil}&\text{(Theorem~\ref{theo:probcondition})}\\
  &\leq p^{-\frac{u}{\ln p}}&\left(\left\lceil\frac{u}{\ln p}\right\rceil\geq \frac{u}{\ln p}\right)\\
  &\leq \enumber^{-u}.&
\end{align*}
Hence
\[
\bbE\ln^k\kappa(\fkf)\leq \Gamma(k+1)\leq k^k,
\]
as we wanted to show.
\end{proof}

\begin{proof}[Proof of Proposition~\ref{prop:padicprobability}]
By the Smith Normal Form and the fact that random $p$-adic vectors are $\mathrm{GL}_N$-invariant, we can assume, without loss of generality, that $A=\begin{pmatrix}\bbI_r|\mathbb{O}\end{pmatrix}$.

Now, write $\fkx=(\fky,\fkz)$. We have that $\fky\in\Zp^r$ and $\fkz\in\Zp^{N-r}$ are inpendent random $p$-adic vectors. Therefore
\begin{align*}
    \bbP\left(\frac{\|x\|}{\|Ax\|}\geq p^{s}\right)&=\bbP(\max\{\|\fky\|,\|\fkz\|\}\geq p^s\|\fky\|)&\\
    &=\sum_{k=0}^{\infty}\bbP(\max\{\|\fky\|,\|\fkz\|\}\geq p^s\|\fky\|,\,\|\fky\|=p^{-k})&\text{(Decomposition in cases)}\\
    &=\sum_{k=0}^{\infty}\bbP(\max\{p^{-k},\|\fkz\|\}\geq p^{s-k},\,\|\fky\|=p^{-k})&\\
    &=\sum_{k=s}^{\infty}\bbP(\|\fkz\|\geq p^{s-k},\,\|\fky\|=p^{-k})&(p^{s-k}>p^{-k})\\
    &=\sum_{k=s}^{\infty}\bbP(\|\fkz\|\geq p^{s-k})\bbP\|\fky\|=p^{-k})&\text{(Independece)}\\
    &=\sum_{k=s}^{\infty}\left(1-\bbP(\|\fkz\|\leq p^{s-k-1})\right)\bbP\|\fky\|=p^{-k})&\\
    &=\sum_{k=s}^{\infty}\left(1-p^{(N-r)(s-k-1)}\right)\left(1-p^{-r}\right)p^{-rk}&\text{(Proposition~\ref{prop:padicprobability})}
\end{align*}
Finally, the proof ends after summing some geometric series and an elementary computation.
\end{proof}

\subsection{Probabilistic analysis of the Strassman count}

We provide probabilistic bounds for the Strassman count at a point. We divide our analysis depending on whether we are counting over the full $\Zp$ or over an smaller closed ball $\overline{B}(x,p^{-s})$. As we will see, the behaviour is very different in these two cases.

\begin{theo}\label{theo:strassmanprob1}
Let $\fkf\in\Zp[T]$ a random $p$-adic polynomial of degree $d$ and $x\in \Zp$. Then: 
\begin{equation}
    \bbP(\St(\fkf;x,1)=\ell)=\frac{\left(1-p^{-1}\right)}{\left(1-p^{-(d+1)}\right)}p^{\ell-d}
\end{equation}
In particular,
\[\bbE\St(\fkf;x,1)=d+\frac{d+1}{p^{d+1}-1}-\frac{1}{p-1}.\]
\end{theo}
\begin{theo}\label{theo:strassmanprob2}
Let $\fkf\in\Zp[T]$ a random $p$-adic polynomial of degree $d$, $x\in \Zp$ and $s\geq 1$. Then: 
\begin{equation}
    \bbP(\St(\fkf;x,p^{-s})\geq\ell)\leq \frac{4}{3}p^{-s\binom{\ell+1}{2}}.
\end{equation}
Moreover, for $k\geq 1$,
\begin{equation}
\bbE\St(\fkf;x,p^{-s})^k\leq 2p^{-s}\left(1+\left(\frac{k}{s\ln p}\right)^{\frac{k}{2}}\right).
\end{equation}
\end{theo}
\begin{cor}\label{cor:strassmanprob3}
Let $\fkf\in\Zp[T]$ a random $p$-adic polynomial of degree $d$ and $s\geq 1$. Then, for $k\geq 1$,
\begin{equation}
\bbE\max_{n=0}^{p^s-1}\St(\fkf;n,p^{-s})^k\leq \bbE\sum_{n=0}^{p^s-1}\St(\fkf;n,p^{-s})^k\leq 2\left(1+\left(\frac{k}{s\ln p}\right)^{\frac{k}{2}}\right)
\end{equation}
\end{cor}
\begin{remark}
Using Theorem~\ref{theo:strassmannexactforcomplex}, we can interpret $\bbE\sum_{n=0}^{p^s-1}\St(\fkf;n,p^{-s})$ as
\begin{equation}
    \#\{\zeta\in\Cp\mid \fkf(\zeta)=0,\,\dist(\zeta,\Zp)\leq p^{-s}\},
\end{equation}
with the roots counted with multiplicity. In this way, we have just shown that for a random $p$-adic polynomial $\fkf\in\Zp[T]$ of degree $d$,
\[
\bbE\#\{\zeta\in\Cp\mid \fkf(\zeta)=0,\,\dist(\zeta,\Zp)\leq p^{-s}\}\leq 2\left(1+\sqrt{\frac{1}{s\ln p}}\right),
\]
if $s\geq 1$. In this way, we have that $\fkf$ has very few roots nearby $\Zp$.
\end{remark}
\begin{remark}
Note that this shows that the Strassman tree's width is very small with very high probability, even though the initial count $\St(\fkf,0,1)$ is as big as it can be---almost $d$.
\end{remark}
\begin{proof}[Proof of Theorem~\ref{theo:strassmanprob1}]
Without loss of generality, we can assume that $x=0$, since by the ultrametric inequality and Theorem~\ref{theo:strassmannexactforcomplex}, $\St(f;0,1)=\St(f;x,1)$ for every $x\in Zp$.

Let $\fkf=\sum_{k=0}^d\fkf_kT^K$. Note that $\St(\fkf;0,p^{-s})=\ell$ means that for $k<\ell$, $|f_k|\leq |f_\ell|$; and for $k\geq \ell+1$, $|f_k|<|f_\ell|$. By conditioning on $|f_\ell|=p^{-a}$, we have that
\begin{multline}\label{eq:argumentstrassmanprob}
\bbP(\St(\fkf;0,1)=\ell)\\=\sum_{a=0}^\infty \bbP\left(|\fkf_0|\leq p^{-\ell},\ldots,|\fkf_{\ell-1}|\leq p^{-\ell},|f_\ell|=p^{-a},|\fkf_{\ell+1}|< p^{-\ell},\ldots,|\fkf_{d}|< p^{-\ell}\right)\\
=\sum_{a=0}^\infty \left(\prod_{k=0}^{\ell-1}\bbP(|\fkf_k|\leq p^{-a})\right)\bbP(|\fkf_\ell|=p^{-a})\left(\prod_{k=\ell+1}^{d}\bbP(|\fkf_k|< p^{-a})\right),   
\end{multline}
since the $\fkf_i$ are independent. Hence
\[
\bbP(\St(\fkf;0,1)=\ell)=\left(1-p^{-1}\right)\sum_{a=0}^\infty p^{\ell-d-a(d+1)}=\frac{\left(1-p^{-1}\right)}{\left(1-p^{-(d+1)}\right)}p^{\ell-d}.
\]
This proves the first equation.

For the final statement, we have that
\[
\bbE\St(\fkf;x,1)=\frac{\left(1-p^{-1}\right)}{\left(1-p^{-(d+1)}\right)}\sum_{\ell=0}^d \ell p^{-\ell-d},
\]
by the equality just proven. Here, an elementary computation gives the desired result.
\end{proof}
\begin{proof}[Proof of Theorem~\ref{theo:strassmanprob2}]
Since translating the variable by $x$ induces a $\mathrm{GL}$-transformation in the space of $p$-adic polynomials, $\fkf$ and $\fkf(x+T)$ have the same random structure. Thus we can assume, without loss of generality, that $x=0$.

If $\St(\fkf,x;p^{-s})\geq \ell$, then we have that for some $i\geq \ell$, we have that for all $j<\ell$, $|\fkf_j|p^{-sj}\leq |\fkf_i|p^{-si}$. Therefore
\begin{equation}
\bbP(\St(\fkf;0,p^{-s})=\ell)=\bbP(\exists i\geq \ell,\,\forall j<\ell,\, |\fkf_j|\leq |\fkf_i|p^{-s(i-j)})\leq \sum_{i=\ell}^d\bbP(\forall j<\ell,\, |\fkf_j|\leq |\fkf_i|p^{-s(i-j)})
\end{equation}
where the last inequality follows from the union bound. Now, conditioning on $|\fkf_i|=p^{-a}$, we have that
\begin{equation}
  \bbP(\forall j<\ell,\, |\fkf_j|\leq |\fkf_i|p^{-s(i-j)})=\sum_{a=0}^\infty  \bbP(\forall j<\ell,\, |\fkf_j|\leq p^{-a-s(i-j)},\,|\fkf_j|=p^{-a}) 
\end{equation}
where, by independence of the $\fkf_k$ and Proposition~\ref{prop:padicprobability},
\begin{multline}
    \bbP(\forall j<\ell,\, |\fkf_j|\leq p^{-a-s(i-j)},\,|\fkf_i|=p^{-a})=\bbP(|\fkf_i|=p^{-a})\prod_{j=0}^{\ell-1}\bbP(|\fkf_j|\leq p^{-a-s(i-j)})\\=\left(1-p^{-1}\right)p^{-(\ell+1)a-s\frac{\ell(2i-\ell+1)}{2}}.
\end{multline}
Hence
\begin{multline*}
\bbP(\St(\fkf;0,p^{-s})=\ell)\leq \left(1-p^{-1}\right)\sum_{i=\ell}^d\sum_{a=0}^\infty p^{-(\ell+1)a-s\frac{\ell(2i-\ell+1)}{2}}\\
=\frac{1-p^{-1}}{1-p^{-(\ell+1)}}\sum_{i=\ell}^d p^{-s\frac{\ell(2i-\ell+1)}{2}}\\
=\frac{\left(1-p^{-1}\right)\left(1-p^{-\frac{s(d-\ell)\ell}{2}}\right)}{\left(1-p^{-(\ell+1)}\right)\left(1-p^{-s\ell}\right)}p^{-s\ell^2+s\ell\frac{\ell-1}{2}}\\=\frac{\left(1-p^{-1}\right)\left(1-p^{-\frac{s(d-\ell)\ell}{2}}\right)}{\left(1-p^{-(\ell+1)}\right)\left(1-p^{-s\ell}\right)}p^{-\frac{s\ell^2+s\ell}{2}}
\end{multline*}
where the equalities are obtained doing geometric sums. Finally, we have that $p\geq 2$, $s\geq 1$ and $\ell\geq 1$, so
\[
\frac{\left(1-p^{-1}\right)\left(1-p^{-\frac{s(d-\ell)\ell}{2}}\right)}{\left(1-p^{-(\ell+1)}\right)\left(1-p^{-s\ell}\right)}\leq \frac{1-1/2}{(1-(1/2)^2)(1-1/2)}=\frac{4}{3}
\]
and the bound on the probability follows.

For the second part, we have that
\begin{equation}\label{eq:expaux0}
  \bbE\St(\fkf;0,p^{-s})^k=\int_{0}^\infty ku^{k-1}\bbP(\St(\fkf;0,p^{-s})\geq u)\,\mathrm{d}u.  
\end{equation}
since $\St(\fkf;0,p^{-s})$ is a positive random variable. 

Now,
\begin{multline}\label{eq:expaux1}
    \int_{0}^1 ku^{k-1}\bbP(\St(\fkf;0,p^{-s})\geq u)=\int_{0}^1 ku^{k-1}\bbP(\St(\fkf;0,p^{-s})\geq 1)\,\mathrm{d}u\\\leq \frac{4}{3}p^{-s}\int_0^1ku^{k-1}\,\mathrm{d}u=\frac{4}{3}p^{-s} 
\end{multline}
and we have that for all $u> 1$,
\[
\bbP(\St(\fkf;0,p^{-s})\geq u)\leq \frac{4}{3}p^{-\frac{su^2}{2}-s},
\]
since $\bbP(\St(\fkf;0,p^{-s})\geq u)=\bbP(\St(\fkf;0,p^{-s})\geq \lceil u\rceil)$. Thus we only have to bound 
\[
\int_{1}^\infty ku^{k-1}\bbP(\St(\fkf;0,p^{-s})\geq u)\,\mathrm{d}u\leq \frac{4}{3}\int_{1}^\infty ku^{k-1}p^{-\frac{su^2}{2}-\frac{\lceil u\rceil}{2}s}\,\mathrm{d}u\leq \frac{4}{3}p^{-s}\int_{1}^\infty ku^{k-1}p^{-\frac{su^2}{2}}\,\mathrm{d}u.
\]
Doing the change of variables $u=\sqrt{\frac{2v}{s\ln p}}$,
\[
\int_{1}^\infty ku^{k-1}p^{-\frac{su^2}{2}}\,\mathrm{d}u\leq \int_{0}^\infty ku^{k-1}p^{-\frac{su^2}{2}}\,\mathrm{d}u=\left(\frac{2}{s\ln p}\right)^{\frac{k}{2}}\Gamma\left(\frac{k+1}{2}\right),
\]
where, by Stirling's estimation~\cite[Eq.~(2.14)]{conditionbook},
\[
\Gamma\left(\frac{k+1}{2}\right)=\frac{k}{2}\Gamma\left(\frac{k}{2}\right)\leq \sqrt{2\pi}\left(\frac{k}{2}\right)^{\frac{k}{2}}\sqrt{\frac{k}{2}}\enumber^{-\frac{k}{2}+\frac{1}{6k}}\leq \frac{3}{2}\left(\frac{k}{2}\right)^{\frac{k}{2}}.
\]
Thus 
\begin{equation}\label{eq:expaux2}
\int_{1}^\infty ku^{k-1}\bbP(\St(\fkf;0,p^{-s})\geq u)\,\mathrm{d}u\leq 2p^{-s}\left(\frac{k}{s\ln p}\right)^{\frac{k}{2}}.
\end{equation}

Putting \eqref{eq:expaux1} and \eqref{eq:expaux2} back in \eqref{eq:expaux0}, we get the desired bound for the expectation.
\end{proof}
\begin{proof}[Proof of Corollary~\ref{cor:strassmanprob3}]
We have that $\bbE\sum_{n=0}^{p^s-1}\St(\fkf;n,p^{-s})^k\leq \sum_{n=0}^{p^s-1}\bbE\St(\fkf;n,p^{-s})^k$, so Theorem~\ref{theo:strassmanprob2} finishes the proof.
\end{proof}

\section{Complexity and precision analysis of \nameref{alg:strassman}}\label{sec:complexity}

We analyze \nameref{alg:strassman}. First, we analyze the algorithm assuming exact arithmetic operation, i.e., working in the BSS~\cite{BCSSbook}. Second, we provide a finite precision analysis in the flat model of \nameref{alg:strassman}. 

\subsection{Correctness}

We show that the algorithm terminates and it is correct as long as the condition number is finite.

\begin{theo}\label{theo:correctness}
Let $f\in\Qp[T]$ be a $p$-adic polynomial of degree $d$. If $\kappa(f)<\infty$, then \nameref{alg:strassman} terminates and it is correct.
\end{theo}
\begin{proof}
Since $\kappa(f)$ is finite, the algorithm must terminate by Proposition~\ref{prop:conditionminimum}. The algorithm is correct, because we select precisely the $B(x,p^{-s})$ for which $\St(f;x,p^{-s})=1$, we discard those $B(x,p^{-s})$ for which $\St(f;x,p^{-s})=0$ and subdivide the rest. This is not affected by the truncation done at each step, due to Proposition~\ref{prop:strassmansubadditive} which guarantees that Strassman count will only go down. Finally, Proposition~\ref{prop:strassmantoalpha} shows that the obtained approximations satisfy the desired properties.
\end{proof}

\subsection{Arithmetic complexity analysis}

The following theorem provides an arithmetic complexity analysis of the algorithm that is input-dependent.

\begin{theo}\label{theo:arithmeticcomplexity}
Let $f\in\Qp[T]$ be a $p$-adic polynomial of degree $d$. Then:
\begin{enumerate}
\item[(d)] The depth of Strassman tree $\mcT_p(f)$ is bounded by
\[
\frac{\log\kappa(f)}{\log p}+1.
\]
\item[(w)] The width of Strassman tree $\mcT_p(f)$ is bounded by
\[
\sum_{n=0}^{p-1}\St(f;n,p^{-1}).
\]
\item[(c1)] The number of arithmetic operations of lines 1--4 is at most
$
\Oh(d).
$
\item[(c2)] The first iteration of the subdivision loop (lines 5--18) of \nameref{alg:strassman} has a deterministic cost of
$\Oh(dp+\sum_{n=0}^{p-1}\St(f;x,p^{-1})^2)$
or an average cost of
$\Oh(d^2\log^3 d\log p)$.
\item[(c3)] All the iterations after the first of the subdivision loop (lines 5--18) of \nameref{alg:strassman} have a deterministic cost of 
$\Oh(\max_{0\leq n\leq p}\St(f;n,p^{-1})^3p)$
and an average cost of
$\Oh(\max_{0\leq n\leq p}\St(f;n,p^{-1})^3\log p)$.
\end{enumerate}
\end{theo}
\begin{cor}\label{cor:arithmeticcomplexity}
Let $f\in\Qp[T]$ be a $p$-adic polynomial of degree $d$. Then \nameref{alg:strassman} takes
\[
\Oh\left(d^2\log^3 d\log p+\log\kappa(f)\left(\sum_{n=0}^{p-1}\St(f;n,p^{-1})\right)\max_{0\leq n\leq p}\St(f;n,p^{-1})^3\right)
\]
arithmetic operation on $f$ on average, if $p> d$; and
\[
\Oh\left(dp+\sum_{n=0}^{p-1}\St(f;x,p^{-1})^2+\log\kappa(f)\left(\sum_{n=0}^{p-1}\St(f;n,p^{-1})\right)\max_{0\leq n\leq p}\St(f;n,p^{-1})^3p\right)
\]
arithmetic operations on $f$, if $p<d$.
\end{cor}
\begin{proof}[Proof of Theorem~\ref{theo:arithmeticcomplexity}]
(d) This follows from Proposition~\ref{prop:conditionminimum}.

(w) At height $s\geq 1$ of $\mcT_p(f)$, we have that the width is bounded by
\[
\#\{n\in\{0,\ldots,p^s\}\mid \St(f;n,p^{-s})>0\}\leq \sum_{n=0}^{p^s-1}\St(f;n,p^{-s})\leq \sum_{n=0}^{p-1}\St(f;n,p^{-1})
\]
where the last inequality follows from Proposition~\ref{prop:strassmansubadditive}.

(c1) This is immediate.

(c2) This follows from Remark~\ref{rem:line7}.

(c3) This follows from Remark~\ref{rem:line7} and the observation that after the first iteration of the subdivision loop, the degree of the polynomials is bounded by $\max_{0\leq i\leq p-1}\St(f;n,p^{-1})$ due to Proposition~\ref{prop:strassmansubadditive}.
\end{proof}
\begin{proof}[Proof of Corollary~\ref{theo:arithmeticcomplexity}]
Note that the first iteration of the loop comes from (c1) and (c2). Once this is done, there are $\frac{\log\kappa(f)}{\log p}\sum_{n=0}^{p-1}\St(f;n,p^{-1})$ node left in the Strassman tree, by (d) and (W), whose cost is given by $(c3)$.
\end{proof}

\subsection{Precision analysis}

For our precision analysis, we will be using a \emph{flat precision model}: we will write all the $p$-adic numbers with $b$ digits of absolute precision, i.e., in the form
\[
a_0+a_1p+a_2p^2+\ldots+a_{b-1}p^{b-1}+O(p^b).
\]
Note that this is the same as projecting $\Zp$ onto $\bbZ/p^b\bbZ$. Thus we can guarantee that the $b$ digits of precision are preserve for addition, multiplication and division by units of $\Zp$. However, precision might lost when we divide by elements of the form
\[
p^ku+O(p^b),
\]
with $u\in\Zp$ a unit. More precisely, we will loss $k$ digits of precision. The following theorem estimates the precision need to guarantee that \nameref{alg:strassman} runs properly.

\begin{theo}\label{theo:bestprecision}
Let $f\in\Zp[T]$ be a $p$-adic integer polynomial of degree $d$. Then we need $1-\frac{\log\|f\|}{\log p}$ $p$-adic digits of precision at the beginning, and $1-\frac{\log\|f\|}{\log p}+\ell s$ digits of $p$-adic precision for computing the descendants $(h;x',p^{-(s+1)};\ell')$ of each appearing $(g;x,p^{-s};\ell)$ (line 10) to guarantee that the output of \nameref{alg:strassman} at $f$ is correct.\eproof
\end{theo}
\begin{proof}
To compute $\|f\|$ we need at least $1-\frac{\log\|f\|}{\log p}$ digits of precision. Now, at each step, we need to compute $\|f(x'+p^{s+1}T)\|$ where $f(x'+p^{s+1}T)$ is truncated to degree $\ell$, where the $p^s$ make us loss at most $s\ell$ digits of precision. Because of this, we need $s\ell$ extra digits of $p$-adic precision.
\end{proof}

\begin{cor}\label{cor:estimateprecision}
Let $f\in\Zp[T]$ be a $p$-adic integer polynomial of degree $d$. Then \nameref{alg:strassman} needs at most
\[1-\frac{\log\|f\|}{\log p}+\max\left\{d,\left(1+\frac{\log\kappa(f)}{\log p}\right)\max_{1\leq i\leq p-1}\St(f;n,p^{-1})\right\}\]
$p$-adic digits to guarentee correctness for $f$.
\end{cor}
\begin{proof}
By Theorem~\ref{theo:bestprecision}, the first step, requires
\[d+1-\frac{\log\|f\|}{\log p}\]
$p$-adic digits, and the $s$th successive subdivision step requires
\[1-\frac{\log\|f\|}{\log p}+s\max_{1\leq i\leq p-1}\St(f;n,p^{-1})\]
$p$-adic digits, since after the first steps all polynomials involve have degree bounded by $\max_{1\leq i\leq p-1}\St(f;n,p^{-1})$, by Proposition~\ref{prop:strassmansubadditive}. Since $s\leq \frac{\log\kappa(f)}{\log p}$, by Proposition~\ref{prop:conditionminimum}, the proof concludes.
\end{proof}

\subsection{Probabilistic complexity}

We can now prove Theorem~\ref{theo:maintheorem}.

\begin{proof}[Proof of Theorem~\ref{theo:maintheorem}]
We only need to combine Theorem~\ref{theo:correctness} with Corollaries~\ref{cor:arithmeticcomplexity} and~\ref{cor:estimateprecision}, and apply to them the probabilistic results of Section~\ref{sec:random}: Proposition~\ref{prop:randompadicpolynomial}, and Corollaries~\ref{cor:probcondition} and~\ref{cor:strassmanprob3}. In order to bound expression of the form
\[
\bbE\kappa(\fkf)^a\left(\sum_{n=0}^{p^s}\St(\fkf;x,p^{-s})^b\right),
\]
we use Cauchy-Schwarz and Jensen inequalities to obtain:
\[
\bbE\ln^a\kappa(\fkf)\left(\sum_{n=0}^{p^s}\St(\fkf;x,p^{-s})^b\right)\leq \sqrt{\bbE\ln^{2a}\kappa(\fkf)}\sqrt{\bbE\sum_{n=0}^{p^s}\St(\fkf;x,p^{-s})^{2b}}.
\]
\end{proof}

\bibliographystyle{plain}
{\small 
\bibliography{biblio}
}

\appendix

\section[Smale's alpha-theory in the ultrametric setting]{Smale's $\alpha$-theory in the ultrametric setting}\label{sec:smalealphatheory}

Smale's $\alpha$-theory guarantees quadratic convergence from an initial point. In this appendix, we develop Smale's $\alpha$-theory in the ultrametric setting. The results here extends the results of Breiding~\cite{breiding2013}, where Smale's $\gamma$-theorem was extend to the ultrametric setting. We follow the development of Dedieu~\cite{dedieubook}, but taking advantage of the ultranorms.

In what follows, $\bbF$ is a non-archimedian complete field of characteristic zero\footnote{Everything here holds if the characteristic is larger than all the degree involved, but we do not aim for general statements.} with (ultrametric) absolute value $|~|$ and $\Pd[n]$ the set of polynomial maps 
\[f:\bbF^n\rightarrow \bbF^n\]
where $f_i$ is of degree $d_i$. In this setting, we will consider on $\bbF^n$ the ultranorm given by
\[
\|x\|:=\max\{x_1,\ldots,x_n\},
\]
its associated distance
\[
\dist(x,y):=\|x-y\|,
\]
and on $k$-multilinear maps $A:(\bbF^n)^k\rightarrow \bbF^q$ the induced ultranorm, which is given by
\begin{equation}
    \|A\|:=\sup_{v_1,\ldots,v_k\neq 0}\frac{\|A(v_1,\ldots,v_k)}{\|v_1\|\cdots\|v_k\|}.
\end{equation}
With these definitions, we can define Smale's parameters.

\begin{defi}[Smale's parameters]
Let $f\in\Pd[n]$ and $x\in\bbF^n$. We define the following:
\begin{enumerate}[(a)]
    \item \emph{Smale's $\alpha$}: $\alpha(f,x):=\beta(f,x)\gamma(f,x)$, if $\diff_x f$ is non-singular, and $\alpha(f,x):=\infty$, otherwise.
    \item \emph{Smale's $\beta$}: $\beta(f,x):=\|\diff_xf^{-1}f(x)\|$, if $\diff_x f$ is non-singular, and $\alpha(f,x):=\infty$, otherwise.
    \item \emph{Smale's $\gamma$}: $\gamma(f,x):=\sup_{k\geq 2}\left\|\diff_xf^{-1}\frac{\diff_x^kf}{k!}\right\|^{\frac{1}{k-1}}$, if $\diff_x f$ is non-singular, and $\gamma(f,x):=\infty$, otherwise.
\end{enumerate}
\end{defi}

Recall that if $\diff_xf$ is non-singular, then the \emph{Newton operator}
\[\newton_f:x\mapsto x-\diff_xf^{-1}f(x)\]
is well-defined. For a point $x$, the \emph{Newton sequence} is the sequence $\{\newton_f^k(x)\}$. Note that this sequence is well-defined (i.e., $\newton_f^k(x)$ makes sense for all $k$) if and only if $\diff_{\newton_x^k(x)}f$ is non-singular, because otherwise $\newton_f^k(x)$ will not be defined for some $k$. Also note that
\[\beta(f,x)=\|x-\newton_f(x)\|,\]
i.e., $\beta$ measures the length of a Newton step.

\begin{theo}[Ultrametric $\alpha$/$\gamma$-theorem]\label{thm:Smalealphatheory}
Let $f\in\Pd[n]$ and $x\in\bbF^n$. Then the following are equivalent:
\begin{enumerate}
    \item[($\alpha$)] $\alpha(f,x)<1$.
    \item[($\gamma$)] $\dist(x,f^{-1}(0))<1/\gamma(f,x)$.
\end{enumerate}
Moreover, if any of the above equivalent conditions holds, then the Newton sequence, $\{\newton_f^k(x)\}$, is well-defined and it converges quadratically to a non-singular zero $\zeta$ of $f$. More specifically, for all $k$, the following holds:
\begin{enumerate}
    \item[(a)] $\alpha(f,\newton_f^k(x))\leq \alpha(f,x)^{2^{k}}$.
    \item[(b)] $\beta(f,\newton_f^k(x))\leq \beta(f,x)\alpha(f,x)^{2^{k}}$.
    \item[(c)] $\gamma(f,\newton_f^k(x))\leq \gamma(f,x)$.
    \item[(Q)] $\|\newton_f^k(x)-\zeta\|=\beta(f,\newton_f^k(x))\leq \alpha(f,x)^{2^{k}}\beta(f,x)<\alpha(f,x)^{2^{k}}/\gamma(f,x)$.
\end{enumerate}
\end{theo}

The proof of the above theorem, relies in the following lemmas and propositions.

\begin{lem}\label{lem:inversegamma}
Let $f\in\Pd[n]$ and $x,y\in\bbF^n$. If $\gamma(f,x)\|x-y\|<1$, then $\diff_y f$ is non-singular and
\[\|\diff_y f^{-1}\diff_xf\|=1.\]
\end{lem}
\begin{prop}[Variations of Smale's parameters]\label{prop:variationSmaleparameters}
Let $f\in\Pd[n]$ and $x,y\in\bbF^n$. If $\gamma(f,x)\|x-y\|<1$, then:
\begin{enumerate}[(a)]
    \item $\alpha(f,y)\leq \max\{\alpha(f,x),\gamma(f,x)\|y-x\|\}$. Moreover, if $\|y-x\|<\beta(f,x)$, $\alpha(f,y)=\alpha(f,x)$.
    \item $\beta(f,y)\leq \max\{\beta(f,x),\|y-x\|\}$. Moreover, if $\|y-x\|<\beta(f,x)$, $\beta(f,y)=\beta(f,x)$.
    \item $\gamma(f,y)=\gamma(f,x)$.
\end{enumerate}
\end{prop}
\begin{prop}[Variations along Newton step]\label{prop:variationNewton}
Let $f\in\Pd[n]$ and $x\in\bbF^n$. If $\alpha(f,x)<1$, then:
\begin{enumerate}[(a)]
    \item $\alpha(f,\newton_f(x))\leq \alpha(f,x)^2$.
    \item $\beta(f,\newton_f(x))\leq \alpha(f,x)\beta(f,x)$.
    \item $\gamma(f,\newton_f(x))=\gamma(f,x)$.
\end{enumerate}
In particular, $\newton_f(\newton_f(x))$ is well-defined.
\end{prop}

\begin{proof}[Proof of Theorem~\ref{thm:Smalealphatheory}]
If $\alpha(f,x)<1$, then, using induction and Proposition~\ref{prop:variationNewton}, we obtain that (a), (b) and (c) hold. But then the sequence $\{\newton_f^k(x)\}$ converges since
\[\lim_{k\to\infty}\|\newton_f^{k+1}(x)-\newton_f^k(x)\|=0\]
and so it is a Cauchy sequence. Finally, (Q) follows from noting that for $l\geq k$
\[
\|\newton_f^l(x)-\newton_f^k(x)\|\leq \alpha(f,x)^{2^{l-k}}\beta(f,\newton_f^k(x))
\]
and taking infinite sum together with the equality case of the ultrametric inequality. In particular, we have $\dist(x,f^{-1}(0))=\|x-\zeta\|=\beta(f,x)<1/\gamma(f,x)$.

For the other direction, if $\gamma(f,x)=\infty$, then $\alpha(f,x)=\infty$ and $\dist(x,f^{-1}(0))<0$ cannot hold. So we focus in the case when $\gamma(f,x)<\infty$. Let $\zeta\in\bbF^n$ be a zero of $f$ such that $\dist(x,\zeta)<1/\gamma(f,x)$. Then
\[
0=f(\zeta)=f(x)+\sum_{k=1}^{\infty}\frac{\diff_x^kf}{k!}(\zeta-x,\ldots,\zeta-x).
\]
Hence
\[
-\diff_xf^{-1}f(x)=\zeta-x+\sum_{k=2}^{\infty}\diff_xf^{-1}\frac{\diff_x^kf}{k!}(\zeta-x,\ldots,\zeta-x).
\]
Now, the higher order terms satisfy that
\[
\left\|\diff_xf^{-1}\frac{\diff_x^kf}{k!}(\zeta-x,\ldots,\zeta-x)\right\|\leq \left(\gamma(f,x)\|\zeta-x\|\right)^{k-1}\|\zeta-x\|<\|\zeta-z\|
\]
and so, by the equality case of the ultrametric inequality,
\[
\beta(f,x)=\|\zeta-x\|<1/\gamma(f,x),
\]
as desired.
\end{proof}
\begin{proof}[Proof of Lemma~\ref{lem:inversegamma}]
We have that
\[\diff_xf^{-1}\diff_yf=\bbI+\sum_{k=1}^{\infty}\diff_xf^{-1}\frac{\diff_x^{k+1}f(y-x,\ldots,y-x)}{k!}.\]
Now, under the given assumption,
\[\left\|\diff_xf^{-1}\frac{\diff_x^{k+1}f(y-x,\ldots,y-x)}{k!}\right\|\leq \left(\gamma(f,x)\|y-x\|\right)^{k-1}<1\]
for $k\geq 2$, and so, by the the ultrametric inequality, $\|\diff_xf^{-1}\diff_yf-\bbI\|<1$. Therefore
\[
\sum_{k=0}^\infty (\bbI-\diff_xf^{-1}\diff_yf)^k
\]
converges, and it does so to the inverse of $\diff_xf^{-1}\diff_yf$. Since, by assumption $\diff_xf$ is invertible, so it is $\diff_yf$. 

Finally, by the invertibility of $\diff_yf$, we have that
\[
\diff_yf^{-1}\diff_xf=\sum_{k=0}^\infty (\bbI-\diff_xf^{-1}\diff_yf)^k,
\]
and so, by the equality case of the ultrametric inequality, $\|\diff_yf^{-1}\diff_xf\|=1$, as desired.
\end{proof}

\begin{proof}[Proof of Proposition~\ref{prop:variationSmaleparameters}]
We first prove (c) and then (b). (a) follows from (b) and (c) immediately.

(c) We note that under the given assumption, for $k\geq 2$,
\begin{equation}\label{eq:ineqmixedgamma1}
    \left\|\diff_xf^{-1}\frac{\diff_y^kf}{k!}\right\|\leq\gamma(f,x)^{k-1}.
\end{equation}
For this, we expand the Taylor series of $\frac{\diff_y^kf}{k!}$ (with respect $y$) and note that its $l$th term is dominated by
\[
\gamma(f,x)^{k+l-1}\|y-x\|^l,
\]
which, by the ultrametric inequality, gives the above inequality. In this way, for $k\geq 2,$
\[
\left\|\diff_yf^{-1}\frac{\diff_y^kf}{k!}\right\|\leq \left\|\diff_yf^{-1}\diff_xf\right\|\left\|\diff_xf^{-1}\frac{\diff_y^kf}{k!}\right\|\leq \gamma(f,x)^{k-1}
\]
by Lemma~\ref{lem:inversegamma} and~\eqref{eq:ineqmixedgamma1}. Thus $\gamma(f,y)\leq \gamma(f,x)$. Now, due to this, the hypothesis $\gamma(f,y)\|x-y\|<1$ holds, and so, by the same argument, $\gamma(f,x)\leq\gamma(f,y)$, which is the desired equality.

(b) Arguing as in (c), we can show that
\begin{equation}\label{eq:ineqmixedbeta1}
    \left\|\diff_xf^{-1}f(y)\right\|\leq \max\{\|\diff_xf^{-1}f(x)+y-x\|,\gamma(f,x)\|y-x\|^2\}
\end{equation}
by noting that the general term (of the Taylor series of $\diff_xf^{-1}f(y)$ with respect $y$) is dominated by $\gamma(f,x)^{k-1}\|y-x\|^k<\gamma(f,x)\|y-x\|^2$. Now, by Lemma~\ref{lem:inversegamma} and~\eqref{eq:ineqmixedbeta1},
\begin{multline*}
    \beta(f,y)\leq \left\|\diff_yf^{-1}\diff_xf\right\|\left\|\diff_xf^{-1}f(y)\right\|\\
    \leq \max\{\|\diff_xf^{-1}f(x)+y-x\|,\gamma(f,x)\|y-x\|^2\}\leq \max\{\beta(f,x),\|y-x\|\}.
\end{multline*}
For the equality case, note that, by the same argument, we have
\[
\beta(f,x)\leq\max\{\beta(f,y),\|y-x\|\}=\beta(f,y)
\]
where the equality on the right-hand side follows from $\beta(f,x)>\|y-x\|$.
\end{proof}
\begin{proof}[Proof of Proposition~\ref{prop:variationNewton}]
(c) follows from Proposition~\ref{prop:variationSmaleparameters} (c).

(b). We use~\eqref{eq:ineqmixedbeta1} in the Proof of Proposition~\ref{prop:variationSmaleparameters}. By~\eqref{eq:ineqmixedbeta1} and Lemma~\ref{lem:inversegamma},
\[
\beta(f,\newton_f(x))\leq \max\{\|D_xf^{-1}f(x)+N_f(x)-x\|,\gamma(f,x)\|\newton_f(x)-x\|\}.
\]
Now, $\newton_f(x)-x=-\diff_xf^{-1}f(x)$, so the above becomes
\[
\beta(f,\newton_f(x))\leq \max\{0,\gamma(f,x)\beta(f,x)^2\},
\]
which gives the desired claim.

(a) follows from combining (b) and (c).
\end{proof}

From the proof of Theorem~\ref{thm:Smalealphatheory}, we can get the following proposition that will be useful later.

\begin{prop}\label{prop:reversealphatheory}
Let $f\in\Pd[n]$ and $x\in\bbF^n$. If for some $r\in (0,1/\gamma(f,x)]$,
\[B(x,r)\cap f^{-1}(0)\neq\varnothing,\]
then
\[\beta(f,x)=\dist(x,f^{-1}(x))< r.\]
\end{prop}
\begin{proof}
Let $\zeta\in B(x,r)\cap f^{-1}(0)$. Under the given hypothesis, we proved that
\[\beta(f,x)=\|\zeta-x\|\]
in the proof of Theorem~\ref{thm:Smalealphatheory}. This is the desired claim.
\end{proof}

\end{document}